\documentclass{amsart}
\usepackage{amsmath,amssymb,amsthm}
\usepackage{tikz-cd}
\usepackage{color}
\usepackage{url}

\newcommand{\acl}{\operatorname{acl}}

\newcommand{\alg}{\operatorname{alg}}

\newtheorem{theorem}{Theorem}[section]
\newtheorem{prop}[theorem]{Proposition}
\newtheorem{cor}[theorem]{Corollary}
\newtheorem{lemma}[theorem]{Lemma}

\newtheorem{example}[theorem]{Example}

\newtheorem{conv}[theorem]{Convention}
\newtheorem{nota}[theorem]{Notation}

\theoremstyle{definition}
\newtheorem{definition}[theorem]{Definition}

\theoremstyle{remark}
\newtheorem{remark}[theorem]{Remark}

\newcommand{\DCFA}{\operatorname{DCF}_{0,m}\operatorname{A}}
\newcommand{\DCF}{\operatorname{DCF}_{0,m}}

\newcommand{\ord}{\operatorname{ord}}
 
\def\Ind{\setbox0=\hbox{$x$}\kern\wd0\hbox to 0pt{\hss$\mid$\hss} \lower.9\ht0\hbox to 0pt{\hss$\smile$\hss}\kern\wd0}
\def\Notind{\setbox0=\hbox{$x$}\kern\wd0\hbox to 0pt{\mathchardef \nn=12854\hss$\nn$\kern1.4\wd0\hss}\hbox to 0pt{\hss$\mid$\hss}\lower.9\ht0 \hbox to 0pt{\hss$\smile$\hss}\kern\wd0}
\def\ind{\mathop{\mathpalette\Ind{}}}
\def\nind{\mathop{\mathpalette\Notind{}}}

\usepackage{amsmath}
\usepackage{graphicx}
\usepackage[colorlinks=true, allcolors=blue]{hyperref}

\title{Differential equations as difference-differential equations}

\author{Wei Li}
\address{State Key Laboratory of Mathematical Sciences, Academy of Mathematics and Systems Science, Chinese Academy of Sciences}
\email{liwei@mmrc.iss.ac.cn}

\author{Thomas Scanlon}
\address{University of California, Berkeley \\
Department of Mathematics \\
Evans Hall \\
Berkeley, CA 94720-3840 \\
USA}
\email{scanlon@math.berkeley.edu}

\begin{document}
\begin{abstract}
We analyze the behavior of systems of algebraic
differential equations when considered as systems of
difference-differential equations, with special emphasis 
on systems which define minimal sets relative
to the theory $\DCF$ of 
differentially closed fields of characteristic zero with
$m$ distinguished commuting derivations.  We show that if 
$p$ is a minimal type relative to 
$\DCF$ defined by a system
of algebraic partial differential equations and inequations and the 
forking geometry on $p$ is trivial (also called 
``disintegrated''), then $p$ remains minimal when 
regarded as a partial type relative to the theory $\DCFA$ of difference-differentially 
closed fields of characteristic zero with 
$m$ commuting derivations in the sense that every completion 
has $\operatorname{SU}$-rank one.  We illustrate this theorem 
by describing in detail the possible 
difference-differential equations consistent with 
differential equations of the form $y' = y^3 + a_2 y^2 + a_1 y + a_0$ where $a_0$, $a_1$, and $a_2$ are 
differential constants.
\end{abstract}
\maketitle

\section{Introduction}
As a rule of thumb, it is difficult for a function to satisfy both a 
differential equation and a difference equation.  Of course, ``difficult'' 
is not the same as ``impossible''.  
As a trivial example, if $f(x) \equiv c$ is 
a constant function, then it satisfies the differential equation $y' = 0$ and 
the difference equation $\sigma(y) = y$ where $\sigma(f)(x) = f(x+1)$.  
For a less trivial example, the function $f(x) = 2^x$ satisfies the difference 
equation $\sigma(y) = 2 y$ and the differential equation 
$y'' y = (y')^2$, or even $y' = (\ln 2) y$ if we permit ourselves 
$\ln 2$ as a parameter, where here we are writing \(y'\) for 
\(\frac{dy}{dx}\).

Precise forms of this 
principle are proven by Adamczewski, Dreyfus, and
Hardouin in~\cite{ADH-JAMS}.   To give a sense
of their results, let us indicate one formulation.
Let $K = \mathbb{C}(x)$ be the field of rational
functions in one variable over the complex numbers
and let $F$ be the field of meromorphic functions on
the complex plane.  We treat $K$ and $F$ as 
differential fields with the derivation 
$\partial = \frac{d}{dx}$ and as difference fields
with the distinguished automorphism 
$\sigma$ defined by $\sigma(f)(x) := f(x+h)$ for
some fixed nonzero complex number $h$. 
Let $F_0\subset F$ be any field extension of $\mathbb C(x)$ setwise fixed
by $\sigma$ in the sense that 
$\sigma(F_0) = F_0$ with $\{f\in F_0: \sigma(f)=f\}=\mathbb C$ and $F_0\cap \mathbb C(x, \{\operatorname{exp}(\lambda x):\lambda \in \mathbb C\})=\mathbb C(x)$.
They show that if $y \in F_0$ satisfies a nontrivial
linear difference equation with coefficients from 
$K$, that is, if there is some $n \in 
\mathbb{Z}_+$ and a sequence 
$a_0, \ldots, a_{n-1}$ of elements of $K$ for 
which $\sigma^n(y) + \sum_{i=0}^{n-1} a_i \sigma^i(y) = 0$, then either $y \in K$ or $y$ satisfies no
algebraic differential equation with coefficients 
from $K$.

These hypertranscendence results on solutions to linear difference 
equations depend on the choice of the difference-differential fields in 
which the solutions are taken in the sense that 
they would be false if we were to allow \(F\) to be 
an arbitrary difference-differential field extension of \(K\).  
Indeed, if $K$ is a field 
equipped with a derivation $\partial:K \to K$ and 
an endomorphism $\sigma:K \to K$ commuting with $\partial$, then 
for any 
positive integer $e > 0$, there is a difference-differential
field  extension $(M,\sigma,\partial) \supseteq (K,\sigma,\partial)$ and 
some $y \in M$ with $\sigma(y) = y $ and $y$ is a generic solution of a
differential equation of order $e$ with coefficients from $K$.  That is, linear difference 
equations are always \emph{consistent} with differential equations of 
arbitrarily high order, though the solutions witnessing the consistency might 
not be found in the standard difference-differential fields.   On the 
other hand, Bustamante has shown in~\cite{Bustamante-thesis} that 
any solution to the 
nonlinear difference equation $\sigma(y) = y^2 + 1$
which is transcendental over a difference-differential field $K$
and is differentially algebraic over $K$ 
must be constant.  Of course,
in the standard difference-differential fields in which the fixed field of 
$\sigma$ is the constant field of $\partial$, even this pair of equations
has only finitely many solutions. 

In this note, we consider a complementary problem: 
if we know that $y$ is a solution to a nontrivial,
possibly nonlinear, algebraic differential equation, 
what further algebraic difference, or even difference-differential, 
equations might it satisfy?    We answer this question in the 
case that the differential equations define a \emph{minimal set}
relative to the theory of differentially closed fields.  We recall the 
relevant definitions and background results in Section~\ref{sec:background}.  

From the theoretical point of view, the results of this paper may be understood
as elaborations on the following observation. 
 Let $K$ be a difference-differential field of characteristic zero equipped with $m$  derivations $\delta_1,\ldots,\delta_m$ and a distinguished ring automorphism $\sigma$,
 all commuting with each other.
Let $X$ be defined by a system 
of differential equations (and, if you like, inequations) over $K$.  If we now regard
$X$ as being defined by a system of difference-differential equations (in which, 
for now, the operator $\sigma$ does not appear), then we may understand 
all further difference-differential equations on solutions to $X$ by 
considering sets $Y$ defined by 
differential equations on the products $Y \subseteq X \times X^\sigma \times 
\cdots \times X^{\sigma^r}$ where for an automorphism 
$\tau$ we are writing $X^\tau$ for the set of solutions to the system of 
differential equations obtained by applying $\tau$ to the coefficients of the 
equations defining $X$.  

This kind of move should be familiar from the conversion 
of systems of algebraic differential equations to systems of purely algebraic 
equations by replacing expressions like $\delta_1^{\alpha_1} \cdots \delta_m^{\alpha_m} x_j$ with 
a new variable $x_{j;\alpha_1, \ldots, \alpha_m}$.  An even more basic instance may be seen 
when one converts a system of algebraic equations into a system of linear equations
by replacing monomials with new variables.   What makes our reduction from 
difference-differential
equations to differential equations so powerful is that, in some cases, 
the structure theory for sets defined by differential equations is so rigid that 
we may deduce that there are very few consistent systems of additional 
differential equations on the products $Y \subseteq X \times X^\sigma \times 
\cdots \times X^{\sigma^r}$.  

We work out this principle in the case that $X$ is minimal.  The 
trichotomy theorem for differentially closed fields \cite{Hrushovski-Sokolović,Pillay-Ziegler-arcs,LS-Zilber-dichotomy}
breaks this problem into 
three cases which essentially reduce to the cases that
\begin{itemize}
    \item $X$ is (related, technically, \emph{nonorthogonal}, to) the field of constants $\mathsf{C}$ 
defined by $\delta_1 x = 
\cdots = \delta_m x = 0$,
    \item $X$ is (nonorthogonal to) a so-called \emph{Manin kernel} defined by differential equations on  a simple abelian variety coming from the Gauss-Manin connection, or 
    \item $X$ is \emph{geometrically trivial},
    also called \emph{disintegrated}, 
    which means that if \(S\) is a set of solutions 
    from \(X\) and each pair of elements of \(S\) is independent, then
    \(S\) is independent as a set.
\end{itemize}

In the first case, if $X = \mathsf{C}$, when regarded as a system of difference-differential equations, then $X$ carries the structure of a difference field and admits a rich collection of further difference-differential equations.  
The nonorthogonality relation 
allows for some further complication, but we will see that in this first case, 
$X$ admits many further difference-differential equations.

The second case may sound the most complicated, but has the most easily described 
structure.  A Manin kernel $X$ carries the structure of a commutative group. 
In analogy to the description of subgroups of abelian varieties defined by 
difference equations given by Hrushovski~\cite[Section 4]{Hr-MM}  
we show that either there are no new difference-differential equations
consistent with $X$ or $X$ admits many new difference-differential 
equations in the sense that its generic type acquires Lascar rank $\omega$
relative to $\DCFA$ and this happens only in the presence of a definable group.

Despite its name, the most interesting case is the third case of a 
geometrically trivial set.   Suppose that $X$ is defined over the 
algebraically closed difference-differential field $K$.
We will observe that while there
may be some additional difference-differential equations consistent with the 
differential equations defining $X$ (in the sense that they could have 
infinitely many solutions) defined over $K$ itself, 
we will not find further consistent difference-differential equations 
defined over larger difference-differential field extensions.
In model theoretic language, our result is that $X$ remains \emph{minimal} as a partial type 
relative to the theory of existentially closed difference-differential fields.

In some cases, geometric triviality of $X$ may be upgraded to \emph{total disintegration} 
which means that distinct solutions of $X$ are necessarily independent.  Under this 
hypothesis on $X$, the possible difference-differential equations consistent with $X$
are extremely restricted, though some may exist.  

In Section~\ref{sec:examples-of-differential-equations},
we work out computational examples exhibiting these 
various phenomena.  For example, we determine the difference-differential equations 
consistent with the differential equations $y' = y^3 + a_2 y^2 + a_1 y + a_0$ where
$a_0$, $a_1$, and $a_2$ are constant parameters and
$y^3 + a_2 y^2 + a_1 y + a_0$ is irreducible over $\mathbb{Q}(a_0,a_1,a_2)$.  

\subsection*{Acknowledgments}
WL is partially supported by NSFC Grants (12671568, 12288201 and 12122118), the Strategic Priority Research Program of the Chinese Academy of Sciences under Grants XDA0480503 and  XDB0640200, National Key R\&D Program of China 2023YFA1009401. TS is partially supported by NSF grants DMS-1800492 and DMS-2201045. 
The AI tools Claude Fable 5.1 and ChatGPT Astra 6 were used to proofread this document, but were not used to conduct the research. 
WL especially thanks James Freitag, L\'{e}o Jimenez, 
Will Johnson, Omar Le\'{o}n S\'{a}nchez, and Ronnie Nagloo
for helpful discussions that deepened her understanding of model-theoretic
notions. We are particularly grateful to
James Freitag for suggesting that we consider the
$\aleph_0$-categorical case, which led to
Theorem~\ref{thm:omega-cat-dichotomy}.

\section{Background on difference-differential algebra and the model theoretic formalism}
\label{sec:background}
In this section we recall some basic  
differential and difference-differential algebra and 
the structure of definable sets, especially 
minimal sets.   The reader will find a 
detailed introduction to the model theory of 
ordinary differential
fields in~\cite{Marker-dcf} and 
to the model theory of partial differential fields in~\cite{Suer2007}.  
Marker's 
slides~\cite{Marker-sm-slides} give a good overview
of the theory of strongly minimal sets in differentially 
closed fields.

\begin{conv}
\label{conv:char-zero}
For the rest of this 
paper, even if it is not explicitly noted, by ``field''
we will mean ``field of characteristic zero''.
\end{conv}

\begin{definition}
\label{def:delta-and-delta-sigma-rings}  
Let $m \in \mathbb{N}_{\geq 1}$ be a natural number.  We write $\Delta = \{ \delta_1, \ldots, \delta_m\}$
for a set of $m$ symbols for derivations.  By a $\Delta$-ring we mean a commutative ring $R$ 
equipped with commuting derivations $\delta_i^R:R \to R$ for $1 \leq i \leq m$.   From now on we 
will simply write $\delta_i:R \to R$ for the derivation on $R$.  A \(\sigma\)-\(\Delta\)-ring 
is a $\Delta$-ring $(R, \delta_1, \ldots, \delta_m)$ given together with a distinguished 
ring endomorphism $\sigma:R \to R$ which commutes with all of the distinguished derivations 
$\delta_i:R \to R$.  A \(\sigma\)-\(\Delta\)-ring $R$ is
\emph{inversive} if $\sigma$ is an automorphism.  A $\Delta$-field, respectively,  (inversive) 
\(\sigma\)-\(\Delta\)-field, is a 
$\Delta$-ring, respectively, 
(inversive)  \(\sigma\)-\(\Delta\)-ring, whose underlying ring is a field.
\end{definition}

\begin{conv}
\label{conv:inversive-difference}
Henceforth, when we say ``\(\sigma\)-ring''
or ``\(\sigma\)-\(\Delta\)-ring''
or ``\(\sigma\)-whatever'', 
we mean \emph{inversive} 
\(\sigma\)-ring, \emph{inversive}
\(\sigma\)-\(\Delta\)-ring, or 
\emph{inversive} \(\sigma\)-whatever.
That is, we require \(\sigma\) to 
be an \emph{automorphism}.
\end{conv}

\begin{remark}
In the language of our 
introduction, we would refer to a $\Delta$-field 
as a differential field and a \(\sigma\)-\(\Delta\)-field
as a difference-differential field. 
Maintaining ``$\Delta$'' 
as part of the notation recalls that we have 
$m$ distinguished derivations. 
\end{remark}

With the following definition we recall the 
constructions of the rings of differential and 
difference-differential polynomials.

\begin{definition}
\label{def:delta-poly}  
Given a $\Delta$-ring \(R\) and a 
tuple $x = (x_1, \ldots, x_n)$ of variables, the 
$\Delta$-ring
of $\Delta$-polynomials over $R$ is the free $\Delta$-$R$-algebra generated by $x$.  We 
denote this algebra  by $R \{ x\}_\Delta$.  As an 
$R$-algebra, $R \{ x \}_{\Delta}$ is the polynomial ring
in the formal variables $\delta_1^{\ell_1} \cdots \delta_m^{\ell_m} x_i$ for $(\ell_1, \ldots, \ell_m) \in 
\mathbb{N}^m$ and $1 \leq i \leq n$.  
From the universal mapping property of a free algebra, 
whenever $R$ is a $\Delta$-ring, $f:R \to S$ is a homomorphism of
$\Delta$-rings, and $a \in S^n$ is an $n$-tuple
from $S$, then there is a unique \(\Delta\)-ring homomorphism
$\overline{f}:R\{x\}_\Delta \to S$
extending $f$ and having $\overline{f}(x_i) = a_i$ for 
$1 \leq i \leq n$.   For $P \in R\{x\}_{\Delta}$,  
we write $P(a) := \overline{f}(P)$.  

When \(R\) is a 
\(\sigma\)-\(\Delta\)-ring, we have a similar construction of the 
ring of \(\sigma\)-\(\Delta\)-polynomials
in the variables \(x\), \(R \{x\}_{\sigma,\Delta}\), as the free 
inversive \(\sigma\)-\(\Delta\)-\(R\)-algebra generated by \(x\).  As an 
\(R\)-algebra, \(R \{x\}_{\sigma,\Delta}\) is the 
polynomial ring over \(R\) in the formal variables $\sigma^{\ell_0} \delta_1^{\ell_1} \cdots \delta_m^{\ell_m} x_i$ for $(\ell_0,\ell_1, \ldots, \ell_m) \in 
\mathbb{Z} \times \mathbb{N}^{m}$ and $1 \leq i \leq n$. 
\end{definition}

The theory of $\Delta$-fields of characteristic zero, $\operatorname{DF}_{0,m}$, may be expressed naturally in the 
language of $\Delta$-rings, $\mathcal{L}(+,-,\cdot,0,1,\delta_1, \ldots, \delta_m)$, which is the 
expansion of the language of rings by 
unary function symbols for the derivations in $\Delta$.  Likewise, the theory of \(\sigma\)-\(\Delta\)-fields
of characteristic zero, $\operatorname{DFA}_{0,m}$, may be expressed in the language 
of \(\sigma\)-\(\Delta\)-rings, $\mathcal{L}(+,-,\cdot,0,1,\delta_1, \ldots, \delta_m,\sigma)$, which is the expansion of
the language of $\Delta$-rings by a function symbol for the distinguished 
automorphism $\sigma$.

The theory of $\Delta$-fields of characteristic zero admits a model companion, the theory $\DCF$
of differentially closed fields of characteristic zero with $m$ commuting derivations~\cite{Robinson-DCF, McGrail-JSL}.
In fact, $\DCF$ admits quantifier elimination as shown by Blum in the case 
of $m = 1$~\cite{Blum-thesis}, 
explicitly by B\"{o}ger and McGrail, 
separately, in general~\cite{Böger-diplom,McGrail-JSL} and Yaffe under the weaker hypothesis that the derivations 
in $\Delta$ generate a Lie algebra of dimension \(m = |\Delta|\)
rather than that they commute~\cite{Yaffe}.  The theory of 
differentially closed fields is characterized by the 
following Nullstellensatz:  a $\Delta$-field $K$ is
 differentially closed if and only if whenever 
$x = (x_1, \ldots, x_n)$ is a finite tuple of variables
and $P_1, \ldots, P_\ell \in K \{ x \}_{\Delta}$ and $Q \in K \{ x \}_\Delta$ are a finite list of $\Delta$-polynomials over 
$K$ for which in some $\Delta$-field extension 
$L$ of $K$, there is a solution $a \in L^n$ to the 
system of algebraic $\Delta$-equations and inequations
$P_1(a) = \cdots = P_\ell(a) = 0$ and $Q(a) \neq 0$,  
then there is such a solution already in $K$.

The theory of \(\sigma\)-\(\Delta\)-fields also has a 
model companion, $\DCFA$, of 
difference-differentially closed fields of 
characteristic zero.  It is reported 
in~\cite{Chatzidakis-generic} that the existence of 
$\operatorname{DCF}_{0,1}\operatorname{A}$ was first observed by 
Hrushovski.  This theory was worked out in 
detail by Bustamante~\cite{Bustamante-thesis}.
Guzy and Rivi\`{e}re obtained a characterization of existentially 
closed difference-differential fields with 
several commuting derivations in~\cite{GR-dcfam} though it is not 
clear on the face of it that their axiomatization may be expressed in 
first-order logic.   Le\'{o}n S\'{a}nchez showed that the 
model companion \(\DCFA\) exists in~\cite{LS-DCFAm}.
Difference-differentially closed fields are characterized
by the analogue of the Nullstellensatz 
describing differentially closed fields, but 
this time with \(\sigma\)-\(\Delta\)-polynomials
in place of mere $\Delta$-polynomials. 
Unlike the theory $\DCF$,
the theory $\DCFA$ does 
not enjoy quantifier elimination.  A weaker
quantifier simplification whereby every 
formula $\varphi$ is equivalent to one of the form $\exists y \psi(x,y)$,
where $\psi$ is quantifier-free and $\psi$ implies that each coordinate of 
$y$ is algebraic over
the difference-differential field generated by $x$ and the parameters,
holds \cite{Bustamante-thesis, LS-DCFAm}. 

While \(\DCF\) is
complete, to determine  completions of 
\(\DCFA\) it is necessary to 
specify the action of \(\sigma\) on the algebraic 
numbers,  up to conjugacy in
$\operatorname{Gal}(\mathbb Q^{\mathrm{alg}}/\mathbb Q)$;
see \cite[Proposition~3.1(ii)]{LS-DCFAm}.
Throughout, we fix one such completion.

\begin{conv}
For the sake 
of simplifying some of our statements and definitions, 
let us fix a universal domain \(\mathbb{U}\) for a chosen
completion of \(\DCFA\).  That is, 
\(\mathbb{U}\) is a model of \(\DCFA\)
and if \(K \subseteq \mathbb{U}\) is an algebraically closed 
``small'', say, \(|K| < |\mathbb{U}|\), \(\sigma\)-\(\Delta\)-subfield of 
\(\mathbb{U}\) and \(K \subseteq L\) is an extension to another
small \(\sigma\)-\(\Delta\)-field then there is an embedding of 
\(L\) into \(\mathbb{U}\) over \(K\).  Forgetting \(\sigma\), 
\(\mathbb{U}\) is also a universal domain for \(\DCF\).
\end{conv}

\begin{nota}
\label{nota:ast-for-delta}
We will have occasion to compare notions defined for \(\Delta\)-fields and
for \(\sigma\)-\(\Delta\)-fields.  Often, these concepts are 
defined in analogous ways.  We will write a simple asterisk, ``\(\ast\)'', 
for ``\(\Delta\)'' or ``\(\sigma\)-\(\Delta\)''.
Thus ``\(\ast\)-field'' is 
to be read as ``\(\Delta\)-field'' or ``inversive \(\sigma\)-\(\Delta\)-field'' as the 
context demands.
\end{nota}

Let us
introduce the \(\ast\)-closure and \(\ast\)-algebraic-closure.

\begin{definition}
    For \(A \subseteq \mathbb{U}\)
we define \(\langle A \rangle_\ast\) to be the \(\ast\)-field
generated by \(A\) in \(\mathbb{U}\).   We define \(\acl^\ast(A)\) to be the 
field-theoretic algebraic closure of \(\langle A \rangle_\ast\). 
For a tuple $a$ from $\mathbb{U}$ and a $\ast$-subfield $K\subseteq\mathbb{U}$ we write $K\langle a\rangle_\ast := \langle K\cup\{a\}\rangle_\ast$.
\end{definition} 

\begin{remark}
The set  \(\acl^\ast(A)\) is the model-theoretic algebraic closure of \(A\) relative to
\(\DCF\) or \(\DCFA\), respectively.  That is, 
\(\acl^\ast(A)\) is the set of elements of \(\mathbb{U}\) which 
belong to some \emph{finite} definable set, defined using parameters
from \(A\), again relative to \(\DCF\) or \(\DCFA\), respectively.
Equivalently, \(\acl^\ast(A)\) is the set of elements of \(\mathbb{U}\)
whose orbit under \(\operatorname{Aut}^\ast(\mathbb{U}/A)\), the
group of automorphisms of \(\mathbb{U}\) as a \(\ast\)-field
fixing \(A\) pointwise, is finite.
\end{remark}

\begin{remark}
Let us note that $\langle A \rangle_\Delta=\mathbb{Q}\left(\delta_1^{\ell_1} \cdots \delta_m^{\ell_m}(a): a\in A, \ell_i\in\mathbb N\right)$ and  \(\langle A \rangle_{\sigma\text{-}\Delta}\) \(= \mathbb{Q}\left(\bigcup_{n \in \mathbb{Z}} \sigma^n(\langle A \rangle_\Delta)\right)\).   
\end{remark}

The theories \(\DCF\) and \(\DCFA\) inherit a good theory of independence from 
the more classical algebraic independence.

\begin{definition}
\label{def:ast-ind}    
For \(A\), \(B\), and \(C\) small subsets of $\mathbb{U}$,
we say that \(A\) is \(\ast\)-free from \(B\) over \(C\), written \(A \ind^\ast_C B\), 
if the field \(\langle A \cup C\rangle_\ast\) is algebraically independent from 
\(\langle B \cup C \rangle_\ast\) over \(\langle C \rangle_\ast\).
\end{definition}

\begin{remark}\label{rk:properties-of-independence}
The independence relations \(\ind^\ast\) for \(\ast = \Delta\) or 
\(\ast=\sigma\text{-}\Delta\) agree with forking independence as 
defined in general simple theories.  See~\cite[Proposition 3.1(v)]{LS-DCFAm}.
Since \(\ind^\ast\) agrees with forking independence in a simple theory, it satisfies some useful properties.
For example, it is symmetric in the sense that for any sets \(A\), \(B\), and \(C \subseteq \mathbb{U}\)
we have \[A {\ind_B}^\ast C \Longleftrightarrow C {\ind_B}^\ast A .\]  In 
our concrete case, it is easy to deduce this symmetry from the evident symmetry of algebraic independence.  
Another useful property is transitivity of independence:  if \(B \subseteq C \subseteq D \subseteq \mathbb{U}\) and 
\(A \subseteq \mathbb{U}\) then we have \[A {\ind_B}^\ast D \Longleftrightarrow (A {\ind_B}^\ast C \text{ and } A {\ind_C}^\ast D) .\]
As with symmetry, transitivity for \(\ind^\ast\) is inherited from transitivity for algebraic 
independence.
\end{remark}

\begin{remark}
\label{rk:general-stable}  
With the exception of the analysis of specific differential equations carried out in 
Section~\ref{sec:examples-of-differential-equations}, almost everything
we do makes sense for expansions of models of general stable theories by automorphisms.   That is, 
if we allow \(\mathbb{U}\) to be a large saturated model with a stable theory and \(\sigma:\mathbb{U} \to 
\mathbb{U}\) to be a generic automorphism of \(\mathbb{U}\), then we may define an independence
relation \(A \ind_B^\sigma C\) on triples of small subsets of \(\mathbb{U}\) by the rule that 
\[A {\ind_B}^\sigma C \Longleftrightarrow \operatorname{acl}(\bigcup_{n \in \mathbb{Z}} \sigma^n(A \cup B)) \ind_{\operatorname{acl}(\bigcup_{n \in \mathbb{Z}} \sigma^n(B))}  \operatorname{acl}(\bigcup_{n \in \mathbb{Z}} \sigma^n(B \cup C)) \]
where here \(\operatorname{acl}\) is the model theoretic algebraic closure in \(\mathbb{U}\) and 
\(\ind\) is forking independence relative to the theory of \(\mathbb{U}\).  As in Remark~\ref{rk:properties-of-independence}
it is easy to check that \(\ind^\sigma\) satisfies symmetry and transitivity.  Since our principal motivation is to understand 
differential equations, our results are written for \(\Delta\)-fields and \(\sigma\)-\(\Delta\)-fields, but the interested reader 
can easily generalize to arbitrary expansions of stable structures by automorphisms by replacing 
\(\ind^\Delta\) with \(\ind\) and \(\ind^{\sigma\text{-}\Delta}\) by \(\ind^\sigma\).
\end{remark}

With the next definition we encode the equations 
satisfied by some tuple with ideals in the 
rings of  
\(\ast\)-polynomials. 

\begin{definition}
\label{def:ideal}
Let $K$ be a $\ast$-field and let $x = (x_1, \ldots, x_n)$
be a tuple of variables.
For a point $a \in \mathbb{U}^n$, we 
let $I_\ast(a/K)$ be the \(\ast\)-ideal in 
\(\mathbb{U} \{ x \}_\ast\) generated by 
\(\{ P \in K \{ x \}_\ast : P(a) = 0 \}\).
For  $S \subseteq \mathbb{U} \{ x \}_\ast$ a set of
$\ast$-polynomials we let  
$\mathbb V(S)(\mathbb{U}) := \{ b \in \mathbb{U}^n : P(b) = 0 
\text{ for all } P \in S \}$.  Usually, we will write 
$X =\mathbb V (S)$ and then $X(K)$ for $\mathbb V(S)(K) := \mathbb V(S)(\mathbb{U}) \cap K^n$.  
We will say 
that $X$ is defined by $S$ and that \(X\) is defined over \(K\) if
\(S \subseteq K \{ x \}_\ast\).
\end{definition}

From the point of view of the first-order languages of 
\(\ast\)-rings,  the 
quantifier-free definable sets in 
\(\ast\)-fields are given by finite 
Boolean combinations of 
\(\ast\)-polynomial equations.   As noted above, 
because \(\DCF\) has quantifier elimination, 
relative to \(\DCF\) these comprise the entire
class of definable sets.   

\begin{conv}
From now on, we use the term ``definable set'' to mean 
``quantifier-free definable set relative to \(\DCF\) or \(\DCFA\)''
as the context demands.
\end{conv}

Let us recall the notion of a type, specialized to the 
theories \(\DCF\) and \(\DCFA\).

\begin{definition}
A \emph{(quantifier-free) partial $\ast$-type} in the
variables $x=(x_1,\ldots,x_n)$ is a small,  
consistent set $\Sigma$ of equations
$\ulcorner P(x)=0\urcorner$ and inequations
$\ulcorner P(x)\neq 0\urcorner$, where
$P\in\mathbb U\{x\}_\ast$ is a $\ast$-polynomial.
Here, consistency means that for any finite subset \(\Sigma_0 \subseteq \Sigma\) 
there is a solution \(a \in \mathbb{U}^n\)
to all the equations \(P(a) = 0\) and inequations \(Q(a) \neq 0\) with 
\(\ulcorner P(x) = 0 \urcorner \in \Sigma_0\) and 
\(\ulcorner Q(x) \neq 0 \urcorner \in \Sigma_0\). 
\end{definition}

With the next definition we see how to relate types to their 
solutions in \(\mathbb{U}\).

\begin{definition}
If \(\Sigma\) is a partial type, then 
there is a solution to all of \(\Sigma\) in 
\(\mathbb{U}\) by our universality hypothesis on 
\(\mathbb{U}\).  We refer to such a 
solution as a \emph{realization} of \(\Sigma\).  We 
write \(\Sigma(\mathbb{U})\) for the set of 
realizations of \(\Sigma\) in \(\mathbb{U}\).  
\end{definition}

\begin{definition}
Let $K\subseteq\mathbb U$ be a small $\ast$-subfield.
A partial $\ast$-type $\Sigma$ is \emph{over $K$} if 
\(P \in K \{x\}_\ast\) for every 
equation \(\ulcorner P(x) = 0 \urcorner \)
and inequation \(\ulcorner P(x) \neq 0 \urcorner \)
in \(\Sigma\).
We say that $\Sigma$ is \emph{complete} if, for every
$P\in K\{x\}_\ast$, exactly one of
$\ulcorner P(x)=0\urcorner$ and
$\ulcorner P(x)\neq 0\urcorner$ belongs to $\Sigma$.
We write $S_{x,\ast}(K)$ for the space of complete quantifier-free
$\ast$-types over $K$ in the variables $x$, or simply $S_{\ast}(K)$ when the tuple $x$ is clear.

For $a\in\mathbb U^n$, we define its $\ast$-type over $K$ by
\[
\begin{aligned}
\operatorname{tp}_\ast(a/K)
&:= \bigl\{
  \ulcorner P(x)=0\urcorner :
  P\in K\{x\}_\ast,\ P(a)=0
  \bigr\}\\
&\quad\cup \bigl\{
  \ulcorner P(x)\neq 0\urcorner :
  P\in K\{x\}_\ast,\ P(a)\neq 0
  \bigr\}.
\end{aligned}
\]
\end{definition}

\begin{remark}
If the partial type \(\Sigma\) is finite, then 
the set of realizations \(\Sigma(\mathbb{U})\) is 
a definable set.  In general, for any partial type
\(\Sigma\), we have that \(\Sigma(\mathbb{U}) = \bigcap_{\Sigma_0 \subseteq \Sigma, \text{ finite }} \Sigma_0(\mathbb{U})\)
expresses the set of realizations of \(\Sigma\) as an intersection of definable sets.  
\end{remark}

\begin{remark}
It follows from the saturation hypothesis on 
\(\mathbb{U}\) that if \(K \subseteq \mathbb{U}\)
is a small \(\ast\)-subfield then every type in 
\(S_{x,\ast}(K)\) is of the form \(\operatorname{tp}_\ast(a/K)\)
for some \(a \in \mathbb{U}^n\).
\end{remark}

\begin{remark}
If $K$ is a small $\ast$-field, there is a natural map
from $S_{x,\ast}(K)$ to the set
$\operatorname{Spec}^{\ast}(K\{x\}_{\ast})$
of prime $\ast$-ideals, given by
\[
\begin{aligned}
\mathfrak{p}\colon S_{x,\ast}(K)
&\longrightarrow
\operatorname{Spec}^{\ast}(K\{x\}_{\ast}),\\
p
&\longmapsto
\bigl\{
P\in K\{x\}_{\ast}
:
\ulcorner P(x)=0\urcorner\in p
\bigr\}.
\end{aligned}
\]
If \(\ast = \Delta\), then \(\mathfrak{p}\) is 
a bijection.  In general, if \(K\) is algebraically
closed, \(\mathfrak{p}\) is still a bijection,
though when \(\ast = \sigma\text{-}\Delta\) and 
\(K\) is not algebraically closed, then because there 
may be extensions of \(\sigma\) to the 
algebraic closure of \(K\) which are not conjugate 
to the specified extension computed from the 
algebraic closure of \(K\) inside of \(\mathbb{U}\), 
some prime \(\sigma\)-\(\Delta\)-ideals
in \(K\{x\}_{\sigma\text{-}\Delta}\) might 
not be in the image of \(\mathfrak{p}\).

If the type \(p \in S_{x,\ast}(K)\) is realized by 
\(a \in \mathbb{U}^n\), so that \(p = \operatorname{tp}_\ast(a/K)\), 
then \(I_\ast(a/K)\) is the ideal generated by \(\mathfrak{p}(p)\) in 
\(\mathbb{U}\{x\}_\ast\).
\end{remark}

We may understand independence in terms of 
the \(\ast\)-polynomial rings.

\begin{prop} \label{prop-nonforking-ideal}
Let \(K \subseteq \mathbb{U}\) 
be an algebraically closed 
\(\ast\)-field, \(n \in \mathbb{N}\), 
\(a \in \mathbb{U}^n\), and 
\(K \subseteq L \subseteq \mathbb{U}\) an 
extension of \(\ast\)-fields.  Then 
\(a \Ind^\ast_K L\) if and only if 
\(I_\ast(a/K) = I_\ast(a/L)\).
\end{prop}
\begin{proof}
Set $F:=K\langle a\rangle_\ast$, $J_K:=\{P\in K\{x\}_\ast:P(a)=0\}$ and
$J_L:=\{P\in L\{x\}_\ast:P(a)=0\}.$ 
Since $K$ is algebraically closed, $a\Ind^\ast_K L$
is equivalent to linear disjointness of
$F$ and $L$ over $K$.   We show that this is equivalent to 
\(J_L=J_K\cdot L\{x\}_\ast\), i.e., \(I_\ast(a/K) = I_\ast(a/L)\) after extending to $\mathbb U\{x\}_\ast$.  Suppose first that $F$ and $L$ are linearly disjoint
over $K$. For $P\in J_L$, choose a $K$-basis
$\lambda_1,\ldots,\lambda_r$ of the span of its
coefficients and write
\(
P=\sum_{i=1}^r\lambda_iP_i, P_i\in K\{x\}_\ast.
\)
By linear disjointness, the $\lambda_i$ remain
linearly independent over $F$.
Thus $P_i(a)=0$ for every $i$, which proves \(J_L=J_K\cdot L\{x\}_\ast\). Conversely, $J_L=J_KL\{x\}_\ast$ implies that a
$K$-basis of $K\{a\}_\ast$ remains linearly independent
over $L$, so, by clearing denominators, $F$ and $L$
are linearly disjoint over $K$.
\end{proof}

Let us recall the definitions of algebraic and 
minimal types.

\begin{definition}
Let \(K\) be a small \(\ast\)-field and \(p \in S_{x,\ast}(K)\) a
complete \(\ast\)-type over \(K\).  We say that \(p\) is
\emph{algebraic} if every component of a realization of \(p\) 
belongs to \(K^{\operatorname{alg}}\).  We say that \(p\) is 
\emph{minimal} if for every \(\ast\)-field \(L\) with \(K \subseteq L \subseteq \mathbb{U}\)
there is a unique non-algebraic type \(q \in S_{x,\ast}(L)\) with \(p \subseteq q\).
\end{definition}

\begin{remark}
\label{remark:alg-minimality}
For \(K\) an algebraically closed 
\(\ast\)-subfield of \(\mathbb{U}\), the type \(p \in S_{x,\ast}(K)\) is minimal if and only if it has infinitely many realizations in 
\(\mathbb{U}\) but for any \(P \in \mathbb{U} \{x\}_\ast \backslash (\mathfrak{p}(p)\cdot \mathbb{U} \{x\}_\ast)\), \(P(x) = 0\) has at most finitely many solutions 
in \(p(\mathbb{U})\).  In particular, if the transcendence degree of \(\operatorname{Frac}(K \{x\}_\Delta/\mathfrak{p}(p))\) 
 over \(K\) is one, then \(p\) is a minimal type. The converse fails as, for instance, the generic type of the third order algebraic differential equation satisfied by the $j$-function is minimal~\cite{FS2018}, although \(
 \operatorname{tr.deg}_K
\operatorname{Frac}\bigl(K\{x\}_\Delta/\mathfrak{p}(p)\bigr)=3.
\)
\end{remark}

With the next definitions we specify what it means for minimal types to be 
orthogonal to each other.

\begin{definition}
Let \(K\) be a small \(\sigma\)-\(\Delta\)-subfield of \(\mathbb{U}\) and \(p \in S_{x,\Delta}(K)\) and \(q \in S_{y,\Delta}(K)\) two
minimal types over \(K\) in the variables \(x = (x_1, \ldots, x_n)\) and \(y = (y_1, \ldots, y_\ell)\), respectively.
We say that \(p\) and \(q\) are \emph{weakly orthogonal}, written \(p \perp^{\operatorname{w}} q\), if whenever \(a \in \mathbb{U}^n\) and \(b \in \mathbb{U}^\ell\) realize \(p\) and \(q\), respectively, then \(a \Ind^\Delta_K b\).    We say that \(p\) and \(q\) are 
\emph{orthogonal}, written \(p \perp q\), if for any small \(\sigma\)-\(\Delta\)-field  extension \(L\) of $K$ and any non-forking extensions
\(\widetilde{p} \in S_{x,\Delta}(L)\) and \(\widetilde{q} \in S_{y,\Delta}(L)\) of \(p\) and \(q\) we have 
\(\widetilde{p} \perp^{\operatorname{w}} \widetilde{q}\).  
\end{definition}

\begin{remark}
\label{rk:alg-non-orthogonal}
We may understand (non)-orthogonality algebraically.  If \(K\) is an algebraically closed 
\(\sigma\)-\(\Delta\)-field and \(p \in S_{x,\Delta}(K)\) and \(q \in S_{y,\Delta}(K)\) 
are two minimal types, then \(p \not \perp^{\operatorname{w}} q\) if and only if there 
are realizations \(a\) and \(b\) of \(p\) and \(q\) so that \(\operatorname{acl}^\Delta(K \cup \{a\}) = \operatorname{acl}^\Delta(K \cup \{b\})\).
\end{remark}

\begin{example}
\label{example:weakly-orthogonal}
Let \(K := \mathbb{Q}^{\operatorname{alg}}\) be the algebraic closure of the rational numbers regarded as a 
\(\sigma\)-\(\Delta\)-subfield of \(\mathbb{U}\).  Let \(m = n = \ell = 1\).  Take \(p := p_{\mathsf{C},K}\) to be  the generic type of the constants.  That is, 
\(p\) is generated by the equation \(\delta_1 x_1 = 0\) and the inequations \(f(x_1) \neq 0\) for every 
nonzero polynomial \(f \in \mathbb{Q}^{\operatorname{alg}}[x_1]\).    Let \(q\) be the type isolated by 
the equation \(\delta_1(y_1) = y_1\) and the inequation \(y_1 \neq 0\).  We see that \(p\) and \(q\) are minimal types from the 
transcendence degree criterion of Remark~\ref{remark:alg-minimality}.

Then \(p \perp^{\operatorname{w}} q\) since if \(a\) is a 
realization of \(p\), then the derivation on the 
field \(\operatorname{acl}^\Delta(\mathbb{Q}^{\operatorname{alg}} \cup \{a \}) = \mathbb{Q}(a)^{\operatorname{alg}}\)
is trivial as \(a\) is a constant, but the realizations of \(q\) are nonconstant.  On the other hand, 
\(p \not \perp q\) for if we take \(L\) to be generated over \(K\) by one solution \(d\) to \(q\), let \(a\) be a
constant which is not algebraic over \(L\), and set \(b := ad\), then \(\operatorname{tp}(b/L)\) is a minimal type extending
\(q\) though visibly \(a \nind^\Delta_L b\).
\end{example}

The main problem we address in this note is: suppose that \(K \subseteq \mathbb{U}\) is a 
small, algebraically closed inversive \(\sigma\)-\(\Delta\)-subfield of \(\mathbb{U}\) and
\(p \in S_{x,\Delta}(K)\) is a minimal type over \(K\) in the sense of \(\DCF\).  What kind of set 
does \(p\) define when regarded as a \emph{partial} type relative to \(\DCFA\)?   We answer the question in detail when \(p\) is 
\emph{disintegrated}.

\begin{definition}
\label{def:trivial} 
Let \(K\) be a small \(\sigma\)-\(\Delta\)-subfield of \(\mathbb{U}\) and \(p \in S_{x,\Delta}(K)\) a minimal 
type over \(K\) relative to \(\DCF\).  We say that \(p\) is \emph{geometrically trivial}, or
\emph{disintegrated}, if every pairwise independent set of
realizations of \(p\) is independent over \(K\).
Explicitly, for any set \(S\) of realizations of \(p\), if
\(a \ind^\Delta_K b\) for all distinct \(a,b \in S\), then
\(
a \ind^\Delta_K \bigl(S \smallsetminus \{a\}\bigr)\)
 for every   \(a \in S.\)
\end{definition}

\begin{remark}
If \(K\) is any small \(\Delta\)-field and \(q \in S_{x,\Delta}(K)\) 
is a geometrically trivial minimal type, then 
\(p_{\mathsf{C},K} \perp q\).  This follows from the facts that geometric triviality is preserved under nonorthogonality of minimal types and  that \(p_{\mathsf C,K}\) is not geometrically trivial.
\end{remark}

A stronger condition on a type which implies geometric triviality is that of being \emph{totally disintegrated}.

\begin{definition}
\label{def:totally-disintegrated}
Let \(K\) be a small \(\sigma\)-\(\Delta\)-subfield of \(\mathbb{U}\) and \(p \in S_{x,\Delta}(K)\) a minimal 
type over \(K\) relative to \(\DCF\).  We say that \(p\) is \emph{totally disintegrated} if
every set of realizations of \(p\) is independent.
\end{definition}

\begin{remark}
In Definition~\ref{def:totally-disintegrated} we could have required merely that the type \(p\) is non-algebraic:  minimality follows
from the requirement that every set of distinct
realizations of $p$ is independent over $K$ relative
to $\DCF$.
\end{remark}

Early in the study of the theory \(\DCF\) from the standpoint of stability theory, totally disintegrated types 
were constructed independently by Kolchin~\cite{Kolchin-constrained},
Rosenlicht~\cite{Rosenlicht}, and Shelah~\cite{Sh-dcf}.  Later in this note, we will
make use of results of Hrushovski and Itai~\cite{HrIt} giving criteria from the 
geometry of algebraic curves to recognize totally disintegrated types.

Let us recall a fundamental theorem
of Baldwin and Harrington 
proven in the 
context of stable theories.

 \begin{lemma}\cite[Lemma 1.4]{Baldwin-Harrington} \label{lemma-trivialregular}
 Let $p\in S(A)$ be trivial and regular.
 If $a$ realizes $p$ and for some $B_0, B_1$ with $B_0\ind_A B_1$, $a\nind_A B_0\cup B_1$, then $a \nind_A B_0$ or $a \nind_A B_1$.
 \end{lemma}

In \cite{Baldwin-Harrington} a stationary type is called \emph{trivial} if any three pairwise independent realizations of any of its nonforking extensions are independent; for minimal types over an algebraically closed $\ast$-field this agrees with Definition~\ref{def:trivial} above (see Proposition~\ref{prop:indep-trivial}).

A minimal type is a particular kind of 
regular type.  For the sake of 
completeness, let us recall the 
definition of \emph{regular type} and 
explain why minimality implies regularity.

\begin{definition}
\label{def:regular}
A type \(p \in S_x(A)\) is 
\emph{regular} if for any parameter 
set \(B \supseteq A\) extending 
\(A\) and pair of realizations \(a\)
and \(b\) of \(p\), if 
\(a \ind_A B\) and \(b \nind_A B\),
then \(a \ind_B b\).  
\end{definition}

In Definition~\ref{def:regular} we 
are taking \(A\) and \(B\) to be 
parameter sets in some model of a
stable theory, the type space 
\(S_x(A)\) would be defined using
all formulas (not just quantifier-free
formulas) and \(\ind\) means 
forking independence.  As we have 
already noted, by quantifier elimination
for \(\DCF\), relative 
to \(\DCF\) our notion of \emph{type}
agrees with the general one and 
\(\ind^\Delta\) agrees with 
forking independence.   Usually, in
the definition of \emph{regular type}
one requires that \(p\) be non-algebraic,
which means for general theories that
it has infinitely many realizations in 
some model.

\begin{prop}
\label{prop:minimal-to-regular}
If \(K\) is a small \(\sigma\)-\(\Delta\)-field and \(p \in S_{x,\Delta}(K)\) is 
a minimal type over \(K\), then 
\(p\) is regular.
\end{prop}
\begin{proof}
Let \(K \subseteq B \subseteq \mathbb{U}\)
and \(a\) and \(b\) be two realizations
of \(p\) where \(B\) is a small set, 
\(a \ind_K^\Delta B\) and \(b \nind_K^\Delta B\).  
As noted before, the independence
\(a \ind_K^\Delta B\) means 
\(a \ind_K^\Delta B\) or from 
the definition \( \langle K \cup \{a\} \rangle_\Delta\) is algebraically 
independent from \(L := \langle B \rangle_\Delta\) over \(K\).  
Since algebraic independence is 
insensitive to replacing a field
by its algebraic closure, we 
have \(a \ind_K^\Delta L^{\operatorname{alg}}\).
On the other hand, since \(b \nind_K^\Delta B\), 
we have that \(b \nind_K^\Delta L\). 
By minimality of \(p\), 
\(\operatorname{tp}(a/L)\) is the
unique nonalgebraic extension of \(p\)
to \(S_{x,\Delta}(L)\). As 
\(\operatorname{tp}(b/L)\) is a
different extension of \(p\) to 
\(S_{x,\Delta}(L)\), it must 
be algebraic.  That is, the 
components of \(b\) belong to 
\(\acl^\Delta(B) = L^{\operatorname{alg}}\).  Thus, 
as \(a \ind_K^\Delta L^{\operatorname{alg}}\)
and \(B \cup \{ b \} \subseteq L^{\operatorname{alg}} \), we conclude \(a \ind_K^\Delta B \cup \{ b \} \), 
and hence \(a \ind_B^\Delta b\) by transitivity or base monotonicity, as
required.
\end{proof}

Let us draw some useful conclusions
from Lemma~\ref{lemma-trivialregular}.

\begin{cor}
\label{cor:trivial-field-of-definition-of-dependence}
Suppose that \(p\) and \(q\) are
trivial regular types over the set
of parameters \(A\), \(A \subseteq B\), 
\(a\) is a realization of \(p\) and 
\(b\) is a realization of \(q\), 
\(a \ind_A B\), \(b \ind_A B\), but
\(a \nind_B b\), then \(a \nind_A b\).
\end{cor}
\begin{proof}
By Lemma~\ref{lemma-trivialregular} applied with $B_0 = B$ and $B_1=\{b\}$, 
if \(a \ind_A b\), because \(a \ind_A B\),
we would have \(a \ind_A B \cup \{b\}\).
But then by monotonicity, 
\(a \ind_B b\), contrary to our 
hypothesis.  Hence, 
\(a \nind_A b\).
\end{proof}

The next proposition allows us to 
extend the characteristic 
property of a trivial type -- that 
pairwise independence implies 
independence of a set of realizations --
from one type to several.

\begin{prop}
\label{prop:indep-trivial}
Let \(A\) be a set of parameters  
and let \(p_i \in S_{x_i}(A)\) be 
trivial, regular types over \(A\)
for \(1 \leq i \leq n\).  Suppose that
\(b_i\) is a realization of \(p_i\)
for \(1 \leq i \leq n\) and that 
\(b_i \ind_A b_j\) for 
\(i \neq j\).  Then \(\{b_1, \ldots, 
b_n \}\) is independent as a set.  
That is, \(b_i \ind_A \{b_j : j \neq i \}\) for each \(i \leq n\).
\end{prop}
\begin{proof}
We prove this result by induction 
on \(n\).  The cases of \(n \leq 2\)
are trivial.  Consider the inductive 
case of \(n+1\).  If 
\( \{b_1, \ldots,b_{n+1}\}\) were
not independent, then possibly 
at the cost of reordering the types, 
we would have \(b_{n+1} \nind_A \{b_1, 
\ldots, b_{n}\}\).  Take \(B_0 := 
\{ b_n \}\) and \(B_1 := \{ b_1, 
\ldots, b_{n-1}\}\). By our 
inductive hypothesis, 
\(B_0 \ind_A B_1\).  We also 
have \(b_{n+1} \ind_A B_0\), by our
hypothesis of pairwise
independence, and \(b_{n+1} \ind_A B_1\)
from the inductive hypothesis.  
Lemma~\ref{lemma-trivialregular}
gives that 
\(b_{n+1} \ind_A B_0 \cup B_1\), 
that is, \(b_{n+1} \ind_A 
\{b_1, \ldots, b_n \}\), contrary
to our supposition.  Hence, 
\(\{b_1, \ldots, b_{n+1}\}\)
is independent and we conclude.
\end{proof}

Let us use 
Proposition~\ref{prop:indep-trivial}
to characterize the possible 
relations between totally disintegrated
types.

\begin{prop}
\label{prop:dependence-totally-disintegrate}
Let \(K \subseteq \mathbb{U}\)
be a small algebraically closed \(\sigma\)-\(\Delta\)-field
and let \(p \in S_{\Delta}(K)\) and 
\(q \in S_{\Delta}(K)\) be two 
minimal, totally disintegrated
types over \(K\) relative to $\DCF$.  If 
\(p \not \perp q\), then there 
are \(\Delta\)-rational maps
\(f\) and \(g\) defined over \(K\) so that
for \(a\) any realization of \(p\), 
\(f(a)\) is a realization of \(q\),  
for \(b\) any realization of \(q\), 
\(g(b)\) is a realization of \(p\), and 
\(a = g(f(a))\) and \(b = f(g(b))\).
Moreover, there is at most one 
such pair of \(\Delta\)-rational 
maps on the types for \(p\) and \(q\).
 \end{prop}
 \begin{proof}
 Let \(L \supseteq K\) be a 
 small \(\Delta\)-field extension of \(K\),
 \(a\) a realization of \(p\), and \(b\)
 a realization of \(q\) witnessing that 
 \(p \not \perp q\).  That is, \(a \ind_K^\Delta L\), 
 \(b \ind_K^\Delta L\), but \(a \nind_L^\Delta b\).
 By Corollary~\ref{cor:trivial-field-of-definition-of-dependence}, 
 \(a \nind_K^\Delta b\).  Thus, by 
 Remark~\ref{rk:alg-non-orthogonal},
 we would have \(\acl^\Delta(K \cup \{a\}) = 
 \acl^\Delta(K \cup \{b\})\). Let us 
 check that, in fact, \(K \langle a \rangle_\Delta = K \langle b \rangle_\Delta\).
 If \(b \notin K \langle a \rangle_{\Delta}\), then 
 there would be some \(b' \neq b\) and
 \(\tau\in\operatorname{Gal}(K \langle a \rangle_{\Delta}^{\operatorname{alg}} / K \langle a \rangle_\Delta)\) such that $\tau(b)=b'$.  As the \(\Delta\)-field structure
 on any \(\Delta\)-field of characteristic
 zero extends uniquely to an algebraic 
 closure, $\tau$ is also an automorphism of 
 \(\Delta\)-fields.  Hence, \(b'\) is also
 a realization of \(q\).  By total disintegration of \(q\), \(b \ind_K^\Delta b'\).  
 However, \( \operatorname{acl}^\Delta(K \cup \{ b \} ) = \operatorname{acl}^\Delta(K \cup \{a \} ) =\tau(\operatorname{acl}^\Delta(K \cup \{ b \} ))= \operatorname{acl}^\Delta(K \cup \{b'\})\) so that \(b \nind_K^\Delta b'\).  
 With this we conclude that \(b\) is a 
 \(\Delta\)-rational function \(f\) of \(a\).  
 The situation is symmetric between
 \(p\) and \(q\) and we see that \(a\) is
 a differential rational function \(g\) of \(b\). 

We see that there cannot be a second differential rational map \(h\) taking realizations of \(p\) to realizations of \(q\) for if there were, then \(\{b, h(g(b)) \}\) 
would be a set of two distinct realizations 
of \(q\) which are dependent over \(K\), 
contrary to total disintegration.  Likewise,
we cannot have another differential
rational function taking realizations
of \(q\) to realizations of \(p\).
 \end{proof}

In many cases of interest, the minimal types we consider come from \emph{strongly 
minimal sets}.

\begin{definition}
\label{def:strongly-minimal}
A definable set $X$ is \emph{strongly minimal} if
$X(\mathbb{U})$ is infinite
and every definable subset is either finite or its complement in 
\(X\) is finite.
\end{definition}

\begin{example}
The field of constants \(\mathsf{C} = \mathbb{U}^\Delta := \{ a \in \mathbb{U} :
\delta_i(a) = 0 \text{ for } 1 \leq i \leq m \}\) is a strongly minimal set
relative to \(\DCF\).
Indeed, by quantifier elimination, the structure induced on \(\mathsf{C}\) by 
differential equations is just that of a pure algebraically closed field
of characteristic zero.  The definable subsets of \(\mathsf{C}\) being 
finite Boolean combinations of the zero sets of ordinary polynomials in one 
variable are all finite or cofinite.  On the other hand, if we consider 
\(\mathsf{C}\) relative to the theory \(\DCFA\) it is \emph{not} strongly minimal.  
For example, the set defined by adjoining the difference equation 
\(\sigma(x) = x\) is infinite as is its complement in \(\mathsf{C}\).
\end{example}

\begin{example}
In the case that $n = 1$ and $m = 1$, 
that is, we are considering 
ordinary differential equations in one variable,
it is easy to 
see that if \(K\) is algebraically closed and 
$P \in K\{x\}_\Delta$ is an irreducible 
$\Delta$-polynomial of order one (that is, $P = F(x,\delta_1 x)$ for some irreducible 
polynomial $F(u,v) \in K[u,v]$ with $\deg_v F > 0$), then the equation 
$P(x) = 0$ defines a strongly minimal set relative to \(\DCF\).
\end{example}

\begin{remark} \label{remark-partialtype}
It is easy to check that if \(X\) is a strongly minimal set defined over the \(\ast\)-field
\(K\), then the partial type \(p_{X,K}\) defined by \(x \in X\) and \(P(x) \neq 0\) for 
each \(P \in K \{x\}_\ast\) for which \(\{ a \in X(\mathbb{U}) : P(a) = 0 \}\) is
finite is minimal.  

In some cases of interest, the formula \(x \in X\) already 
defines a complete minimal type.  For example, this is true when we take 
\(K = \mathbb{Q}\), \(m = 1\), and \(X\) defined by \(\delta_1(x) = x^3 - x^2 \land x \neq 0 \land x \neq 1\).  
\end{remark}

Our notion of a totally disintegrated 
minimal type is related to, though distinct from, 
that of a \emph{strictly minimal set} 
defined in~\cite{Cherlin1985}
and used extensively in~\cite{HrIt}:
 
\begin{definition} \label{def:strictly-minimal-set}
A definable set \(X\), defined over a parameter set \(A\), is \emph{strictly minimal} if it is strongly minimal, its induced structure is \(\aleph_0\)-categorical, meaning that for each \(n\) there are only finitely many \(A\)-definable subsets of \(X^n\), and \(X\) admits no   \(A\)-definable equivalence relations other than equality and the universal relation.
\end{definition}

\begin{lemma}
\label{lm-strictminimal}
Work in a sufficiently saturated model of $\mathrm{DCF}_{0,m}$.
Let $X$ be strictly minimal over $A$, and let $p_{X,A}$
be its generic type over $A$.
Then $p_{X,A}$ is isolated by the formula defining $X$
and is totally disintegrated.
\end{lemma}

\begin{proof}
Suppose that $Y$ is a nonempty proper $A$-definable subset
of $X$. Then
\[
E_Y(x,y)
\quad\Longleftrightarrow\quad
\bigl(x\in Y \land y\in Y\bigr)
\lor
\bigl(x\notin Y \land y\notin Y\bigr)
\]
defines an equivalence relation on $X$ that is neither
equality nor the universal relation, contradicting strict
minimality. Thus $X$ has no nonempty proper $A$-definable
subset, so the formula defining $X$ isolates $p_{X,A}$.

Moreover, \(
X\cap\operatorname{acl}^{\Delta}(A)=\varnothing.
\)
Indeed, an element of this intersection would belong to
a nonempty finite $A$-definable subset of $X$, a contradiction.

We also use the fact that an $\aleph_0$-categorical strongly
minimal set in $\mathrm{DCF}_{0,m}$ is geometrically trivial (see~\cite[Lemma 2.23]{Nagloo–Pillay-2017} whose proof extends to 
several commuting derivations using~\cite[Corollary 4.2(ii)]{Pillay-foundational}). For $x,y\in X$, define
\[x\sim y
\quad\Longleftrightarrow\quad
y\in\operatorname{acl}^{\Delta}(A,x).
\]
Since every element of $X$ is nonalgebraic over $A$,
 $\sim$ is an equivalence relation.  
By $\aleph_0$-categoricity, this  relation
is $A$-definable, and each of its classes
\(
\operatorname{acl}^{\Delta}(A,x)\cap X
\)
is finite. As $X$ is infinite, $\sim$ is not the universal
relation. Strict minimality therefore implies that
$\sim$ is equality and for each $x\in X,$
\(\operatorname{acl}^{\Delta}(A,x)\cap X=\{x\}\). 
Since $p_{X,A}$ is geometrically trivial, pairwise independence of a set of distinct realizations implies its independence.
Thus,  $p_{X,A}$ is totally disintegrated.
\end{proof}

As above, if we take a larger parameter set 
\(B \supseteq A\), then \(X\) might not be 
strictly minimal when regarded as 
a \(B\)-definable set, as, for example, 
there may be elements of \(X\) that are
algebraic over \(B\), but the generic 
type \(p_{X,B}\) of \(X\) over \(B\) is 
still totally disintegrated. Indeed, if $a\neq b$ realize $p_{X,B}$ and $a\nind_B b$, then $a\nind_A b$  
by Corollary~\ref{cor:trivial-field-of-definition-of-dependence}, contradicting total disintegration over $A$.

\section{Preservation of minimality for trivial types}

\vskip5pt
The following theorem is our main result in this section:
\begin{theorem} \label{thm-trivialminimal}
Suppose $K$ is an algebraically closed  \(\sigma\)-\(\Delta\)-field.
Let $p \in S_{x,\Delta}(K)$ be a  minimal, trivial
type relative to \(\DCF\). Then 
relative to \(\DCFA\), every completion of \(p\) 
is minimal as a quantifier-free type. 
\end{theorem}
\begin{remark} Looking ahead to the discussion of Lascar rank in Section~\ref{sec:non-trivial-types} and the characterization of forking relative to \(\DCFA\), the conclusion of Theorem~\ref{thm-trivialminimal} may be rephrased by saying that every completion of \(p\) has \(\operatorname{SU}\)-rank one. \end{remark}

We require three simple lemmas before we prove Theorem~\ref{thm-trivialminimal}.

\begin{lemma}
\label{lem:aut-preserves-minimality-and-triviality}
Let \(K\) be a sub-\(\sigma\)-\(\Delta\)-field of
\(\mathbb{U}\), \(a\) a 
tuple from \(\mathbb{U}\), and
\(N \in \mathbb{Z}\) an integer.  
If \(\operatorname{tp}_{\Delta}(a/K)\) is minimal, then so is \(\operatorname{tp}_{\Delta}(\sigma^N(a)/K)\).
Likewise, if  \(\operatorname{tp}_{\Delta}(a/K)\) is trivial, then so is \(\operatorname{tp}_{\Delta}(\sigma^N(a)/K)\).
\end{lemma}
\begin{proof}
These invariance properties follow easily from the fact that \(\sigma\) preserves 
\(\ind^\Delta\).  Let us elaborate, considering first the minimality.
Let \(L\) be a \(\Delta\)-field extending \(K\) and suppose that 
\(\sigma^N(a) \nind_K^\Delta L\).  Because \(\sigma\), and thus also \(\sigma^N\), is 
an automorphism of \(\mathbb{U}\) as a \(\Delta\)-field, we have 
\(a \nind_K^\Delta \sigma^{-N}(L)\).  Since \(\operatorname{tp}(a/K)\) is minimal, it must be 
that \(a\) is algebraic over \(\sigma^{-N}(L)\).  Applying 
\(\sigma^N\) and observing that \(\sigma^N\) preserves algebraicity, we see that 
\(\sigma^N(a)\) is algebraic over \(L\).  Likewise, because \(a\) is not algebraic over 
\(K\), neither is \(\sigma^N(a)\).  That is, we have demonstrated that \(\operatorname{tp}(\sigma^N(a)/K)\)
is minimal.

The argument for triviality is similar:  suppose that \(\{a_1, \ldots, a_n\}\) is a finite 
list of realizations of \(\operatorname{tp}_{\Delta}(\sigma^N(a)/K)\) and that \(a_i \ind_K^\Delta a_j\) for 
\(1 \leq i < j \leq n\).  Then \(\sigma^{-N}(a_i)\) is a realization of \(\operatorname{tp}(a/K)\) 
for each \(i\) and \(\sigma^{-N}(a_i) \ind_K^\Delta \sigma^{-N}(a_j)\) for 
\(1 \leq i < j \leq n\) as \(\sigma^{-N}\) respects \(\ind_K^\Delta\).  Since
\(\operatorname{tp}_{\Delta}(a/K)\) is trivial, the set \(\{ \sigma^{-N}(a_1), \ldots, 
\sigma^{-N}(a_n) \}\) is \(K\)-independent.  Applying \(\sigma^{N}\) we
see that \(\{a_1, \ldots, a_n\}\) is also \(K\)-independent, verifying that \(\operatorname{tp}_{\Delta}(\sigma^{N}(a)/K)\)
is trivial.
\end{proof}

\begin{lemma}
\label{lem:independentsets-and-forking}
Let \(K \subseteq L\) be an extension of  \(\Delta\)-fields. 
Suppose that \(a_1, \ldots, a_n\) is a finite sequence of tuples from \(\mathbb{U}\) so that for each \(i\), 
\begin{itemize} 
\item \(\operatorname{tp}(a_i/K)\in S_{\Delta}(K)\) is trivial and minimal, and 
\item the tuple \(a_i\) is not algebraic over \(L\). 
\end{itemize}
Then \(a_1, \ldots, a_n \ind_K^\Delta L\). If, in addition, \(\{a_1,\ldots,a_n\}\) is \(K\)-independent, then it is also \(L\)-independent.
\end{lemma}
\begin{proof}
We first argue by induction on \(n\) to show that \(a_1, \ldots, a_n \ind_K^\Delta L\).
As each  \(a_i\) is not algebraic over \(L\) and \(p_i=\operatorname{tp}(a_i/K)\in S_{\Delta}(K)\) is minimal, $a_i\ind_K L$. 
Let us consider the inductive case of \(k+1\).
Our inductive hypothesis is that \(a_1, \ldots, a_k \ind_K^\Delta L\) for $1\leq k<n$.
There are two subcases to consider: \(a_{k+1} \in \operatorname{acl}^\Delta(K 
\cup \{a_1, \ldots, a_k\})\) or  \(a_{k+1} \notin \operatorname{acl}^\Delta(K 
\cup \{a_1, \ldots, a_k\})\).
In the first case, since independence is insensitive to taking algebraic
closures, we have that  \(a_1, \ldots, a_k \ind_K^\Delta L\) implies
that \(\operatorname{acl}^\Delta(K 
\cup \{a_1, \ldots, a_k\}) \ind_K^\Delta L\). Since in this 
case \(\{a_1, \ldots, a_k, a_{k+1}\} \subseteq \operatorname{acl}^\Delta(K 
\cup \{a_1, \ldots, a_k\})\), we conclude by monotonicity that  
\(a_1, \ldots, a_{k+1} \ind_K^\Delta L\).  In the nonalgebraic case, since
\(p_{k+1}\) is minimal, we have \(a_{k+1} \ind_K^\Delta a_1, \ldots, a_k\).  
By the hypothesis of the lemma,   \(a_{k+1} \ind_K^\Delta L\).
By our inductive hypothesis we have \(a_1, \ldots, a_k \ind_K^\Delta L\).
Thus, taking \(B_0 := \{ a_1, \ldots, a_k\}\) and \(B_1 := L\) in 
Lemma~\ref{lemma-trivialregular}, we conclude that 
\(a_{k+1} \ind_K \{ a_1, \ldots, a_k \} \cup L\).   By transitivity, 
this implies that \(a_{k+1} \ind_{K \cup \{a_1, \ldots, a_k \}}^\Delta L\).
Using our inductive hypothesis that \(a_1, \ldots, a_k \ind_K^\Delta L\) and 
transitivity again, we conclude \(a_1, \ldots, a_{k+1} \ind_K^\Delta L\), 
as required.  So by induction, we conclude that \(a_1, \ldots, a_{n} \ind_K^\Delta L\).

Assume in addition that $\{a_1,\ldots,a_n\}$ is $K$-independent. 
Fix $i$, and let \(b:=(a_j)_{j\neq i}.\)
Then \(a_i,b\ind_K^\Delta L.\)
By symmetry and transitivity,
\(
a_i\ind_{K\langle b\rangle_\Delta}^\Delta L.
\)
By the $K$-independence of the set $\{a_1,\ldots,a_n\}$,
$a_i\ind_K^{\Delta}K\langle b\rangle_\Delta$.
By transitivity, \(a_i\ind_K^\Delta L\langle b\rangle_\Delta.\)
By base monotonicity, \(a_i\ind_L^\Delta L\langle b\rangle_\Delta\).
Thus, for every $i$,
\(
a_i\ind^\Delta_L  L\langle a_j:j\neq i\rangle_\Delta.
\)
Hence $\{a_1,\ldots,a_n\}$ is $L$-independent.
\end{proof}

We need one more lemma.

\begin{lemma}
\label{lem:still-non-algebraic}
Let \(K \subseteq \mathbb{U}\) be 
a small  \(\sigma\)-\(\Delta\)-field,
and let \(p\in S_\Delta(K)\) 
be a non-algebraic type.  Suppose that \(q\) is a
completion of \(p\) relative to \(\DCFA\).  Then 
\(q\) is also non-algebraic.
\end{lemma}

\begin{proof}
If \(q\) were algebraic, then it would have a realization \(a\) in 
\(\operatorname{acl}^{\sigma\text{-}\Delta}(K)\). 
Since \(\operatorname{acl}^{\sigma\text{-}\Delta}(K)\) is the 
field theoretic algebraic closure of the inversive \(\sigma\)-\(\Delta\)-field
generated by \(K\) and \(K\) is already an  inversive \(\sigma\)-\(\Delta\)-field, we have \(\operatorname{acl}^{\sigma\text{-}\Delta}(K) = 
\operatorname{acl}^\Delta(K) = K^{\operatorname{alg}}\).  That is, 
the coordinates of \(a\) would lie in \(K^{\operatorname{alg}}\), meaning that 
\(p\) would be algebraic, contrary to our hypotheses.
\end{proof}

\vskip5pt
Let us return to the proof of Theorem~\ref{thm-trivialminimal}.

\begin{proof}
Let $q$ be a completion of $p$ relative to $\DCFA$.
By Lemma \ref{lem:still-non-algebraic}, $q$  is non-algebraic.
We need to show that if \(a\) is a realization of 
\(q\) and \(L\) is an algebraically closed \(\sigma\)-\(\Delta\)-field
extending \(K\) inside \(\mathbb{U}\), then 
either the coordinates of \(a\) come from \(L\) or 
\(a \ind_K^{\sigma\text{-}\Delta} L\). 

Let us suppose towards a contradiction that 
\(a\) is not algebraic over \(L\) while 
\(a \nind_K^{\sigma\text{-}\Delta} L\).
Recall that 
this dependence means that 
\(\langle K \cup a \rangle_{\sigma\text{-}\Delta}\)
and \(L\) are algebraically dependent over \(K\).
Using the facts that \(\langle K \cup a \rangle_{\sigma\text{-}\Delta}\) is 
the field generated by 
\(\bigcup_{j \in \mathbb{Z}} \sigma^j(\langle K \cup \{a \}\rangle_\Delta) \) and an algebraic 
dependence is witnessed by a single polynomial
relation, we can find some finite \(N\) and \(M\) so that 
\(K \langle \sigma^{-M}(a), \ldots, \sigma^{N}(a) \rangle_\Delta\) and \(L\)
are algebraically dependent over \(K\).  Since 
\(\sigma\) restricts to an automorphism of 
both \(K\) and \(L\) and \(\sigma\) respects 
algebraic (in)dependence, we may replace \(N\) with \(N+M\)
so that \(K \langle a, \sigma(a), \ldots, \sigma^N(a)\rangle_\Delta\) is algebraically dependent with \(L\)
over \(K\).

From the definition of \(\ind^\Delta\), this dependence 
means that \(a, \sigma(a), \ldots, \sigma^N(a) \nind_K^\Delta L\)
which contradicts Lemma~\ref{lem:independentsets-and-forking} as 
each type \(\operatorname{tp}(\sigma^i(a)/K)\) is minimal and trivial
and no \(\sigma^i(a)\) is algebraic over \(L\) for 
\(i \in \mathbb{Z}\).  
\end{proof}

In the following, we refine Theorem~\ref{thm-trivialminimal} by assuming the stronger hypothesis that the types we consider are totally disintegrated relative to $\DCF$. We will obtain a 
description of the possible \(\sigma\)-\(\Delta\)-equations consistent 
with a given totally disintegrated type.   Before doing so, we need to introduce some notation.

\begin{definition}
\label{def:twisted-composition}
Let \(f\) be a \(\Delta\)-rational function and \(\tau:K \to K\) a 
\(\Delta\)-field automorphism.  We write \(f^\tau\) for the \(\Delta\)-rational function obtained by applying \(\tau\) to the coefficients of \(f\).  For \(j \in \mathbb{N}\)
we define \(f^{\diamond^\tau j}\) by recursion on \(j\): 
\[f^{\diamond^\tau 0} := f, \qquad  f^{\diamond^\tau j + 1} := f^{\tau^{j+1}} \circ f^{\diamond^\tau j}. \]
\end{definition}

\begin{remark} \label{rk-iteration-of-f}
The notation may seem somewhat unusual, but it will be useful when we study possible relations between a totally disintegrated type
$p=\operatorname{tp}_{\Delta}(a/K)$
and its transforms
$p^{\tau^j}:=\operatorname{tp}_{\Delta}(\tau^j(a)/K)$
where \(\tau\) is an automorphism of \(\mathbb{U}\) as 
a \(\Delta\)-field.
Since \(p^{\tau^j}\) is also totally disintegrated for every $j$, Proposition~\ref{prop:dependence-totally-disintegrate} implies that, if
$p\not\perp p^\tau,$
then there exists a bijective \(\Delta\)-rational function
$f: p(\mathbb U)\longrightarrow p^\tau(\mathbb U).$
For each \(j\geq 0\), its transform \(f^{\tau^j}\) is a bijective $\Delta$-rational map
$f^{\tau^j}:\,p^{\tau^j}(\mathbb U)\longrightarrow p^{\tau^{j+1}}(\mathbb U).$
Thus, defining $f^{\diamond^\tau (j+1)}
:=f^{\tau^{j+1}}\circ f^{\diamond^\tau j},$
we obtain a bijective \(\Delta\)-rational function
$f^{\diamond^\tau j}:
p(\mathbb U)\longrightarrow p^{\tau^{j+1}}(\mathbb U).$
\end{remark}

With Definition~\ref{def:twisted-composition} in place, we may describe the possible \(\sigma\)-\(\Delta\)-equations consistent with a totally disintegrated type $p$.  
We begin with the case where $p$ is orthogonal to all its
nontrivial transforms under powers of $\sigma$.
Recall that an element $a\in\mathbb U$ is
\emph{$\sigma$-transcendental over $K$} if, for every
$\ell\geq 0$, the elements
$a,\sigma(a),\ldots,\sigma^\ell(a)$ are algebraically
independent over $K$.

\begin{prop}
\label{prop:automatic-sigma-hypertranscendence}
Let $K$ be an algebraically closed
$\sigma$-$\Delta$-field, and let
$p\in S_{x,\Delta}(K)$ be a minimal type relative
to $\DCF$. Suppose that
\(p\perp p^{\sigma^i}\) for every  \(i>0.\)
Then, for every realization $a\models p$, the family
$(\sigma^i(a))_{i\in\mathbb Z}$ is independent over
$K$ relative to $\ind^\Delta$. In particular, if $x$ is a
single variable, every realization of $p$ is
$\sigma$-transcendental over $K$.

Moreover, $p$ has a unique complete quantifier-free
$\sigma$-$\Delta$-extension over $K$, and this extension
is minimal relative to $\DCFA$.
\end{prop}

\begin{proof}
Suppose towards a contradiction that there exists $a\models p$ such that the family
$(\sigma^i(a))_{i\in\mathbb Z}$ is not independent over
$K$ relative to $\ind^\Delta$.   Choose such a realization $a$ so that the least positive integer
$\ell$ for which $a,\sigma(a),\ldots,\sigma^\ell(a)$ are $K$-dependent relative to $\ind^\Delta$ is as small as possible.
Then we would have \(a \ind_K^\Delta \sigma(a),
\ldots, \sigma^{\ell-1} a\) and \(\sigma^\ell(a) \ind_K^\Delta \sigma(a),
\ldots, \sigma^{\ell-1} a \) as $\sigma(K)=K$,
whereas
\(
\sigma^\ell(a) \nind_K^\Delta a, \sigma(a),\ldots,\sigma^{\ell-1}(a).
\)
By transitivity and symmetry, this gives \(a \nind_{K \langle \sigma(a), \ldots, \sigma^{\ell-1}(a) \rangle_\Delta}^\Delta \sigma^{\ell}(a)\).  Hence, \(p \not \perp p^{\sigma^\ell}\),
contrary to our hypothesis. Thus, we have proved that for every realization $a\models p$, the family
$(\sigma^i(a))_{i\in\mathbb Z}$ is independent over
$K$ relative to $\ind^\Delta$. In particular, if $a$ is a
single element in $\mathbb U$,  
 $a$ is $\sigma$-transcendental over $K$.

We now prove the assertion about quantifier-free type extensions.
Let $a\models p$. Every quantifier-free $\sigma$-$\Delta$-formula in $a$ can be written as a $\Delta$-formula in finitely
many transforms $(\sigma^i(a))_{i\in I}$ for some finite $I\subseteq\mathbb Z$.
By stationarity and the independence of $(\sigma^i(a))_{i\in I}$ proved above, 
$\operatorname{tp}_{\Delta}\left((\sigma^i(a))_{i\in I}/K\right)$ is uniquely determined by the individual types $p^{\sigma^i}$.
These types depend only on $p$ and the  action
of $\sigma$ on $K$. Hence all realizations of $p$
have the same quantifier-free
$\sigma$-$\Delta$-type over $K$, proving uniqueness
of the complete quantifier-free extension, say $q$.

To prove minimality, let $L\supseteq K$ be a small
algebraically closed $\sigma$-$\Delta$-field, and let
$a\models p$ be nonalgebraic over $L$. Set $p_L=\operatorname{tp}_\Delta(a/L)$.
 By minimality and  stationarity of $p^{\sigma^i}$,  
$p_L^{\sigma^i}=\operatorname{tp}_\Delta(\sigma^i(a)/L)$ is the
unique nonforking extension of $p^{\sigma^i}$ to $L$.
In particular,
\(
\sigma^i(a)\ind^\Delta_K L\) for every \(i\in\mathbb Z.
\) 
Orthogonality is preserved under nonforking
extensions, so
\(p_L\perp p_L^{\sigma^i}\)  for every \(i>0.\)
By the independence assertion proved above with $L$ in place of $K$, we get
$(\sigma^i(a))_{i\in\mathbb Z}$ is independent over $L$.
Consequently, for every finite $I\subseteq\mathbb Z$,
the tuple $(\sigma^i(a))_{i\in I}$ is independent over
$L$. Since $\sigma^i(a)\ind_{K}L$ for each $i$, by transitivity,
\(
(\sigma^i(a))_{i\in I}\ind^\Delta_K L.
\)
By finite character and
\cite[Proposition~3.1(v)]{LS-DCFAm}, we  obtain
\(
a\ind^{\sigma\text{-}\Delta}_K L.
\)
This holds for arbitrary small
$\sigma$-$\Delta$-field extensions of $K$, by passing
to their algebraic closures.
This proves the minimality of $q$.
\end{proof}

We now restrict attention  to the case that \(p\) is nonorthogonal to one of its transforms.

\begin{prop}
\label{prop:completions-totally-disintegrated}
Let \(p \in S_\Delta(K)\) be a totally disintegrated type over a small algebraically closed $\sigma$-$\Delta$-field $K$.
Let \(\ell \in \mathbb{Z}_+\) be the least positive integer 
for which \(p \not \perp p^{\sigma^\ell}\), and \(f\) be the  
\(\Delta\)-rational function over $K$ given by Proposition~\ref{prop:dependence-totally-disintegrate}
so that for \(a\) realizing \(p\) we have that \(f(a)\) realizes \(p^{\sigma^\ell}\). 
Then for any $a\in p(\mathbb U)$,
either 
\begin{itemize}
    \item[(1)] $a$ is $\sigma$-transcendental over $K$; or
    \item[(2)] $a$ satisfies the $\sigma$-equation  $\sigma^{\ell n}(x) = f^{\diamond^{\sigma^\ell}n-1}(x)$ for some $n\in\mathbb N_{>0}$.
\end{itemize}
\end{prop}

\begin{proof}
Suppose \(a\) is not \(\sigma\)-transcendental. 
Then \(a\) is \(\sigma\)-algebraic.
Let \(A \in \mathbb{Z}_+\) be minimal so that \(a, \ldots, \sigma^{A} a\) are not \(K\)-independent relative to \(\ind^\Delta\).
By triviality and minimality of 
\(p\), and hence also of \(p^{\sigma^i}\) for every \(i \in \mathbb{Z}\), this dependence yields \(\sigma^i(a) \nind_K^\Delta \sigma^j(a)\) for some \(0 \leq i < j \leq A\).  Applying \(\sigma^{-i}\) we may take \(i = 0\) and 
by minimality of \(A\), we have \(j = A\). That is, $\sigma^A(a)\in\acl(K,a)$ and $p\not \perp p^{\sigma^A}$.   By Proposition~\ref{prop:dependence-totally-disintegrate},
there is some bijective (on \(p\)) \(\Delta\)-rational function \(h\) taking realizations of \(p\) to realizations of \(p^{\sigma^A}\).
Since both $h(a)$ and $\sigma^A(a)$ are realizations of the totally disintegrated type $p^{\sigma^A}$ in $\acl(K,a)$ with $a$ a realization of a minimal type $p$, 
we have \(h(a) = \sigma^A(a)\).  

We claim that \(A\) is a multiple of \(\ell\). Since $A\geq \ell$, we can write \(A = \ell n + r\) with 
\(0 \leq r < \ell\) and \(n \in \mathbb{Z}_{+}\). Suppose towards
a contradiction that \(r \neq 0\).
By Remark \ref{rk-iteration-of-f}, $f^{\diamond^{\sigma^\ell} n-1}$ is a bijective \(\Delta\)-rational function from $p(\mathbb U)$ to $p^{\sigma^{\ell n}}(\mathbb U)$.
Let \(g\) be the \(\Delta\)-rational
function which is inverse to \(f^{\diamond^{\sigma^\ell} n-1}\) on \(p^{\sigma^{\ell n}}\) given by 
Proposition~\ref{prop:dependence-totally-disintegrate}.  
Then \((h \circ g)^{\sigma^{-\ell n}}\)
is a \(\Delta\)-rational function which takes realizations of \(p\) to 
realizations of \(p^{\sigma^{r}}\).  Hence, \(p \not \perp p^{\sigma^r}\).  As 
\(\ell\) was minimal with \(p \not \perp p^{\sigma^\ell}\), we see that we must have
had \(r = 0\). So $A=\ell n$.

From the uniqueness assertion in Proposition~\ref{prop:dependence-totally-disintegrate},  
we see that \(h = f^{\diamond^{\sigma^\ell} n-1}\).  That is, \(a\) satisfies \(\sigma^{\ell n}(x) = f^{\diamond^{\sigma^\ell}n-1}(x)\).
\end{proof}

\vskip5pt

When $\sigma$ fixes the base field $K$ pointwise, the possible $\sigma$-equations consistent with a totally disintegrated type in Proposition \ref{prop:completions-totally-disintegrated} have a simpler form.

\begin{cor} \label{cor-strictdisintegrated-fixed}
Let $K$ be a  $\sigma$-$\Delta$-field with $K=K^{\sigma}:=\{x\in K: \sigma(x)=x\}$.
Let $p\in S_\Delta(K)$ be a totally  disintegrated type. Then for any $a\in p(\mathbb U)$, either $a$ is $\sigma$-transcendental or $a$ satisfies  $\sigma^\ell(a)=a$ for some $\ell \in\mathbb N_{>0}$.
\end{cor}
\begin{proof}
Since $K=K^{\sigma}$, $p^{\sigma^\ell}=p$ for each $\ell\in\mathbb N_{>0}$. 
Thus, if for some \(a\) realizing \(p\) we have a dependence 
amongst \(a, \sigma(a), \ldots, \sigma^N(a)\) for some \(N\), then 
we have a dependence amongst realizations of \(p\) itself.  Since
\(p\) is totally disintegrated, this would force 
\(\sigma^i(a) = \sigma^j(a)\) for some \(0 \leq i < j \leq N\).  Setting
\(\ell := j - i\), we would have \(\sigma^\ell(a) = a\).
\end{proof}

\begin{example}
Let $p\in S_{\Delta}(\mathbb Q)$ be the generic type of $x'=x^3-x^2.$ By \cite[p.532, Corollary]{Rosenlicht} for \(n=2\) and \cite[p.160]{Kolchin-constrained} in general, if $x_1,\ldots,x_n$ are distinct non-algebraic solutions of $x'=x^3-x^2$, then $x_1,\ldots,x_n$ are algebraically independent over $\mathbb Q.$ 
Thus, $p$ is totally disintegrated.
For each $a\in p(\mathbb U)$, $\sigma^i(a)\in p(\mathbb U)$. Thus, either $a$ is $\sigma$-transcendental over $\mathbb Q$, or $a$ satisfies $\sigma^n(a)=a$ for some $n\in\mathbb N_{>0}$.
\end{example}

We say that a type $p\in S_{\Delta}(K)$ is
$\aleph_0$-categorical over $K$ if, for every $r\geq1$,
there are only finitely many relatively $K$-definable
subsets of $p(\mathbb U)^r$.
Here a relatively $K$-definable subset is the intersection
of $p(\mathbb U)^r$ with a $K$-definable set.

We can also prove an analogue of Corollary~\ref{cor-strictdisintegrated-fixed} for $\aleph_0$-categorical minimal types relative to $\DCF$.

\begin{theorem}
\label{thm:omega-cat-dichotomy}
Let $K$ be a $\sigma$-$\Delta$-field such that
$K=K^{\sigma}$.
Let $p\in S_{x,\Delta}(K)$ be a minimal type in a single variable $x$ that is
$\aleph_0$-categorical over $K$.
For every $\alpha\in p(\mathbb U)$, exactly one of
the following holds:
\begin{enumerate}
    \item $\alpha$ is $\sigma$-transcendental over $K$;
    \item $\sigma^N(\alpha)=\alpha$ for some integer $N\geq1$.
\end{enumerate}
\end{theorem}

\begin{proof}
Let $D:=p(\mathbb U)$.
Since $\sigma$ fixes $K$ pointwise and commutes with
the derivations, $\sigma(D)=D$.
For every finite $B\subseteq D$, the set
\(
\operatorname{acl}^{\Delta}(K\cup B)\cap D
\)
is finite. Indeed, $\aleph_0$-categoricity implies that
there are only finitely many relatively
$(K\cup B)$-definable subsets of $D$.
Every element of \(\operatorname{acl}^{\Delta}(K\cup B)\cap D\) belongs to a finite
such subset, so their union is finite.

Suppose that $\alpha$ is not $\sigma$-transcendental
over $K$. Then there exists $r\geq 1$ such that
\(\alpha,\sigma(\alpha),\ldots,\sigma^r(\alpha)
\)
are algebraically dependent over $K$, and hence dependent
over $K$ relative to $\ind^\Delta$.
Write $b_i:=\sigma^i(\alpha)$, and choose the least
$\ell\geq1$ such that
\(b_0,b_1,\ldots,b_\ell
\)
is dependent over $K$ relative to $\ind^\Delta$.
Such an $\ell$ exists and is positive because $p$
is nonalgebraic.

By the choice of $\ell$, the tuple
$(b_0,\ldots,b_{\ell-1})$ is independent over $K$.
Applying $\sigma$, we see that
$(b_1,\ldots,b_\ell)$ is also independent over $K$.
Since every $b_i$ realizes the minimal type $p$,
the dependence of the full tuple implies
\[
\begin{aligned}
b_\ell
&\in
\operatorname{acl}^{\Delta}(K,b_0,\ldots,b_{\ell-1}),\\
b_0
&\in
\operatorname{acl}^{\Delta}(K,b_1,\ldots,b_\ell).
\end{aligned}
\]
Consequently,
\[
\operatorname{acl}^{\Delta}(K,b_0,\ldots,b_{\ell-1})
=
\operatorname{acl}^{\Delta}(K,b_1,\ldots,b_\ell).
\]

Let
\(
F:=
\operatorname{acl}^{\Delta}(K,b_0,\ldots,b_{\ell-1})
\cap D.
\)
The set $F$ is finite and contains $\alpha$.
Moreover,
\[
\sigma(F)
=
\operatorname{acl}^{\Delta}(K,b_1,\ldots,b_\ell)
\cap D=
\operatorname{acl}^{\Delta}(K,b_0,\ldots,b_{\ell-1})
\cap D=F.
\]
Thus $\sigma$ permutes the finite set $F$, and hence
$\sigma^N(\alpha)=\alpha$ for some $N\geq1$.
\end{proof}

We conclude this section by giving an example which is trivial but not $\aleph_0$-categorical. 

\begin{example}
Consider the differential equation satisfied by the
classical $j$-function:
\[
\chi(y):=
\left(\frac{y''}{y'}\right)'
-\frac12\left(\frac{y''}{y'}\right)^2
+\frac{y^2-1968y+2654208}
       {2y^2(y-1728)^2}(y')^2
=0.
\]
By \cite{FS2018}, its solution set $X$ is strongly
minimal and geometrically trivial.

Fix $\lambda\in\mathbb C$, and let $\sigma$ act by
translation on the differential field generated
over $\mathbb C$ by the functions
$j(t+k\lambda)$, $k\in\mathbb Z$, interpreted as
meromorphic germs. Thus $\sigma$ fixes $\mathbb C$
pointwise, commutes with $\delta=d/dt$, and satisfies
\(\sigma^r(j(t))=j(t+r\lambda),\,\, 
 r\in\mathbb Z.
\)

\begin{enumerate}
\item
If $\lambda\notin\mathbb Q$, then $j(t)$ is
$\sigma$-transcendental over $\mathbb Q$.
 
Hence by Theorem \ref{thm-trivialminimal},  $j(t)$ is $\sigma$-transcendental over any $\sigma$-$\delta$-field extension of $\mathbb Q$ over which it is transcendental.
\item
If $\lambda=p/q\in\mathbb Q$, where
$p\in\mathbb Z$, $q\geq1$, and $\gcd(p,q)=1$,
then the identity $j(t+1)=j(t)$ gives
\(
\sigma^q(j(t))=j(t+p)=j(t).
\)
Moreover, each pair $j(t)$ and $\sigma^r(j(t))$
satisfies a modular relation
\(
\Phi_N\bigl(j(t),\sigma^r(j(t))\bigr)=0
\)
for an appropriate classical modular polynomial
$\Phi_N\in\mathbb Z[Y,Z]$.
\end{enumerate}
\end{example}

\section{Concrete examples of differential equations}
\label{sec:examples-of-differential-equations}

In this section we illustrate 
Theorem~\ref{thm-trivialminimal} with some 
concrete examples of differential equations, analyzed
using the Hrushovski-Itai~\cite{HrIt} formalism of 
differential equations presented by an algebraic curve 
with a differential form.  

For this section, we work with a single derivation,
$\Delta=\{\delta\}$.  Since there is only one 
derivation, we sometimes write  \(y'\) for 
\(\delta y\).  As usual, let $\mathbb U$ be a
universal domain for a fixed completion of $\operatorname{DCF}_{0,1}\operatorname{A}$.
Its reduct to the language of differential rings is a saturated
$\delta$-closed field of characteristic zero.  
We denote its field of constants by
\[
\mathsf C:=\mathbb U^\delta
=\{a\in\mathbb U:\delta(a)=0\}.
\]
For a $\sigma$-$\delta$-subfield $K\subseteq\mathbb U$,
we write $C_K:=K^\delta=K\cap\mathsf C$ for its
field of constants.

\subsection{Results of Hrushovski and Itai}
\label{sec:results-HrIt}

In~\cite{HrIt},
Hrushovski and Itai study
 strictly minimal sets associated with rational $1$-forms 
 on algebraic curves and establish criteria for strict minimality.

\begin{definition}
\label{def:sets-differential-forms}
Let \(K \subseteq \mathsf{C}\) be a small algebraically closed field of constants. 
Let \(X\) be a smooth projective algebraic curve over \(K\) and 
let \(\omega\) be a nonzero rational 1-form on \(X\).  We set 
\(\Xi(X,\omega) := \{ a \in X(\mathbb{U}) :  \omega(a)(\delta a) = 1 \} \)
where, if \(a\) is a pole of \(\omega\), the equation 
\(\omega(a)(\delta a) = 1\) is taken to mean \(\delta a = 0\).
\end{definition}

As argued in \cite[p. 4274]{HrIt},
$\Xi(X, \omega)$ is Kolchin closed and irreducible
and the points of $\Xi(X, \omega)$ algebraic 
over $K$ are precisely the poles of $\omega$.
 
 Let $\Xi^*(\mathbb{P}^1, \omega)$ denote the strongly minimal set obtained by removing  the finitely many algebraic points of $\Xi(\mathbb{P}^1, \omega)$, which are precisely the poles of $\omega$.

\begin{lemma} \cite[Lemma 2.23]{HrIt}  \label{hrushovski2003-lm2.23}  Let $\omega$ be a rational $1$-form on $\mathbb{P}^1$, defined over a constant field $K$. Assume $\omega$ has at least two nonzero residues, and that no two distinct nonzero residues of $\omega$ are rational multiples of each other. Then $X = \Xi^*(\mathbb{P}^1, \omega)$ is strictly minimal over $K$. 
\end{lemma}

As mentioned in \cite[p.~4283]{HrIt} without proof, the nonorthogonality relation between two strictly minimal sets living on $\mathbb{P}^1$
 is governed by M\"obius transformations. For convenience and later use, we include a simple proof of this fact. 
 \begin{lemma}  \label{lm-PGL2}
Let $X=\Xi^*(\mathbb{P}^1,\omega_1)$ 
and $Y=\Xi^*(\mathbb{P}^1,\omega_2)$ 
be two strictly minimal subsets living on $\mathbb P^1$,
where \(\omega_1\) and \(\omega_2\) are 
rational \(1\)-forms on $\mathbb P^1$, defined over a constant field $K$.
 Then $X$ and $Y$ are  non-orthogonal if and only if one
is the image of the other under an element of the definable group
$\operatorname{PGL}_2(\mathbb U)$.
\end{lemma} 

\begin{proof} 
If $X$ and $Y$ are  nonorthogonal, then
there is a relatively Kolchin closed subset 
\(Z \subseteq \Xi^*(\mathbb{P}^1,\omega_1) \times \Xi^*(\mathbb{P}^1,\omega_2)\) 
which projects dominantly and finite-to-one in both directions.  
We may take \(Z\) to be irreducible.  By strict minimality, 
we may take \(Z\) to be defined over the same 
algebraically closed set of parameters as \(X\) and \(Y\).
From strict minimality of these sets, \(Z\) must be the graph of a bijection as if 
\(a, b \in \Xi^*(\mathbb{P}^1,\omega_1)\) and \(c \in \Xi^*(\mathbb{P}^1,\omega_2)\)
with \((a,c) \in Z\) and \((b,c) \in Z\), then \(b\) would lie in a 
finite set definable from \(a\), which by strict minimality can consist only of 
\(a\) itself.  Hence, \(Z\) is a single-rooted relation. 
Reasoning with the 
two factors reversed, we see that \(Z\) is the graph of a function.  
In general, a definable function is piecewise differential rational.
By strong  minimality of $X$, one piece is defined on a cofinite definable subset of $X$ and the remaining pieces consist of algebraic points. 
As $X$ is strictly minimal, $X$ has no algebraic points, and thus the definable function is a single differential rational function. Since the derivative of any point in \(\Xi^*(\mathbb{P}^1,\omega_1)\) is 
rational over that point, it must be that \(Z\) is the graph of an invertible rational function, that is, of a rational automorphism of \(\mathbb{P}^1\): an element of \(\operatorname{PGL}_2(\mathbb{U})\). 
\end{proof}

In the statements below we say that the generic type over $K$ of a definable set $V^*$ is \emph{strictly minimal} if $V^*$ is strictly minimal, in the sense of Definition~\ref{def:strictly-minimal-set}, over some subfield of $K$ over which it is defined.  By Lemma~\ref{lm-strictminimal} and the remark following it, the generic type of $V^*$ over $K$ is then totally disintegrated.

\begin{lemma} \label{lemma-strictlyminimal}
Let $K$ be an algebraically closed $\delta$-field of characteristic
zero. Let $f_1, f_2\in C_K[y]$ have the same degree $n\geq 3$, and for $i=1,2$, write
\(
f_i(y)=f_{i,n}y^n+f_{i,n-1}y^{n-1}+\cdots+f_{i,0}
\).
Let $p_1$ and $p_2$ be the generic types over $K$ of 
\[
V_1^*=\{y:\delta y=f_1(y),\ f_1(y)\neq 0\},
\qquad
V_2^*=\{y:\delta y=f_2(y),\ f_2(y)\neq 0\},
\]
respectively, and suppose that they are strictly minimal.
Then $p_1$ and $p_2$ are nonorthogonal if and only if there exist
$a\in C_K^\times$ and $b\in C_K$ with $a^{n-1}=\frac{f_{1,n}}{f_{2,n}}$ such that
\[f_2(ay+b)=a f_1(y).
\]
In particular,
\(
b=\frac{1}{n}\left(
      a\frac{f_{1,n-1}}{f_{1,n}}-\frac{f_{2,n-1}}{f_{2,n}}
   \right).
\)
Moreover, if $u\models p_1$, $v\models p_2$, and $u\nind_K^\delta v$, then such $a,b$ can be chosen with $v=au+b$.
\end{lemma}

\begin{proof}
By Lemma~\ref{lm-PGL2},  $p_1$ and $p_2$ are nonorthogonal if and only if there is a M\"obius transformation
\(
\phi(y)=\frac{\alpha y+\beta}{\gamma y+\eta}\) with $\alpha, \beta, \gamma, \eta\in \mathbb U$ and
\(\alpha\eta-\beta\gamma\neq 0,
\)
which carries a generic solution of $\delta y=f_1(y)$ to a generic
solution of $\delta z=f_2(z)$. Take a solution $y$ of $\delta y=f_1(y)$ transcendental over $E:=K\langle \alpha, \beta, \gamma, \eta\rangle$.
Differentiating $\phi(y)$ and clearing denominators gives the following
identity in $E[y]$:
\[
\begin{aligned}
& (\gamma y+\eta)^{n-2}
  \Bigl[
  \bigl((\delta\alpha)y+\delta\beta+\alpha f_1(y)\bigr)
       (\gamma y+\eta) \\
&\hspace{35mm}
  - (\alpha y+\beta)
    \bigl((\delta\gamma)y+\delta\eta+\gamma f_1(y)\bigr)
  \Bigr] \\
&\hspace{15mm}=
  (\gamma y+\eta)^n
  f_2\!\left(\frac{\alpha y+\beta}{\gamma y+\eta}\right).
\end{aligned}
\]
The expression in brackets is
$(\alpha\eta-\beta\gamma)f_1(y)+R(y)$, where $\deg R\leq 2$.
If $\gamma\neq 0$, the left-hand side therefore has degree $2n-2$
and leading coefficient
$f_{1,n}\gamma^{n-2}(\alpha\eta-\beta\gamma)\neq 0$.
The right-hand side has degree at most $n$.
Since $2n-2>n$, this is impossible. Hence $\gamma=0$.

Set $a=\alpha/\eta$ and $b=\beta/\eta$.
The above identity now becomes
\[
a f_1(y)+(\delta a)y+\delta b=f_2(ay+b).
\]
Comparing coefficients of $y^n$ gives
$a^{n-1}=f_{1,n}/f_{2,n}$. Since $K$ is algebraically closed and
$f_{1,n}/f_{2,n}\in C_K$, we have $a\in C_K^\times$.
Comparing coefficients of $y^{n-1}$ gives
\[
a f_{1,n-1}=a^{n-1}(n f_{2,n} b+f_{2,n-1}),
\]
and therefore
\[
b=\frac{1}{n}\left(
      a\frac{f_{1,n-1}}{f_{1,n}}-\frac{f_{2,n-1}}{f_{2,n}}
   \right)\in C_K.
\]
Thus $\delta a=\delta b=0$ and $f_2(ay+b)=a f_1(y)$.

Now suppose that $u\models p_1$, $v\models p_2$, and
$u\nind_K^\delta v$. Minimality gives
\[
\operatorname{acl}^\delta(K,u)
=\operatorname{acl}^\delta(K,v)
\] by Remark \ref{rk:alg-non-orthogonal}.
For an affine bijection $\phi(y)=ay+b$ as above, $\phi(u)\models p_2$
and
\[
\operatorname{acl}^\delta(K,\phi(u))
=\operatorname{acl}^\delta(K,u)
=\operatorname{acl}^\delta(K,v).
\]
Strict minimality of $p_2$ therefore implies $v=\phi(u)$.
\end{proof}

\begin{theorem} \label{thm-strictmin}
Let $K$ be an algebraically closed $\sigma$-$\delta$-field of
characteristic zero, and let
\(f(y)=y^n+f_{n-1}y^{n-1}+\cdots+f_0\in C_K[y],
n\geq 3.
\) 
Suppose that the generic type in $S_\delta(K)$ of
\(
V_f^*=\{y:\delta y=f(y),\ f(y)\neq 0\}
\)
is strictly minimal. Let $\xi$ be a solution of
$\delta y=f(y)$ transcendental over $K$.
Then either $\xi$ is $\sigma$-transcendental over $K$, or there exist
$s\geq 1$, $a\in(C_K)^\times$, and $b\in C_K$ such that
\[
\sigma^s(\xi)=a\xi+b,
\]
and
\[
f^{\sigma^s}(y)
=a f\!\left(\frac{y-b}{a}\right).
\]
In the latter case, 
 $a^{n-1}=1$ and $b=\frac{a f_{n-1}-\sigma^s(f_{n-1})}{n}.
$
\end{theorem}
\begin{proof}
Suppose that $\xi$ is $\sigma$-algebraic over $K$.
Since $\xi$ is transcendental over $K$, there is an integer $k\geq 1$
such that $\xi,\sigma(\xi),\ldots,\sigma^k(\xi)$ are algebraically
dependent over $K$.
For $i\geq 0$, write
\[
\xi_i=\sigma^i(\xi),
\qquad
p_i=\operatorname{tp}_\delta(\xi_i/K).
\]
As $\sigma(K)=K$ and $\sigma$ commutes with $\delta$, each $p_i$
is the strictly minimal generic type of
$\delta y=f^{\sigma^i}(y)$ over $K$.
These types are geometrically trivial. Hence
Proposition~\ref{prop:indep-trivial} yields indices
$0\leq i<j\leq k$ such that
$\xi_i\nind_K^\delta\xi_j$.
This dependence is also ordinary algebraic dependence over $K$, for $\delta\xi_r=f^{\sigma^r}(\xi_r)$ implies 
\(K\langle\xi_r\rangle_\delta=K(\xi_r)\).
Thus there exists a nonzero $G\in K[y,z]$ such that $G(\xi_i,\xi_j)=0$.

Set $s=j-i\geq 1$. Applying $\sigma^{-i}$ gives
\(
G^{\sigma^{-i}}(\xi,\sigma^s(\xi))=0\),
where $G^{\sigma^{-i}}\in K[y,z]\setminus\{0\}$.
Since both $\xi$ and $\sigma^s(\xi)$ are transcendental over $K$,  $p_0$ and $p_s$
are not weakly orthogonal.
Apply the final assertion of Lemma~\ref{lemma-strictlyminimal}
to this particular pair of realizations, with $f_1=f$ and $f_2=f^{\sigma^s}$.
Since both polynomials are monic of degree $n$, it yields
$a\in(K^\delta)^\times$ and $b\in K^\delta$ with the desired form satisfying
\[
a^{n-1}=1,
\qquad
\sigma^s(\xi)=a\xi+b,
\qquad
f^{\sigma^s}(a y+b)=a f(y).
\]      
 \end{proof}

\subsection{Computational examples: $y'=y^3+a_2y^2+a_1y+a_0$}

In this section, we study concrete examples arising from autonomous differential equations of the form $$y'=f(y) .$$   We first recall Rosenlicht's criterion characterizing when the generic type of $y'=f(y)$ is orthogonal to the constants.

\begin{theorem}\cite{Rosenlicht} \label{Rosenlicht}
  Let $K$ be a differential field with constant field $C_K$, 
  and let $C_K(y)$ be a pure transcendental extension field of $C_K$. 
 Let $f(y)\in C_K(y)\backslash\{0\}$ and consider the differential equation $y' = f(y)$.  Suppose that $ \frac{1}{f}$ is neither of the forms  
\begin{equation} \label{eq-partialfraction}
\quad \frac{\partial g(y)}{\partial y}, \quad \text{or} \quad c\cdot\frac{\frac{\partial h(y)}{\partial y}}{h(y)},
\end{equation}
 where $g \in C_K(y)$, $c \in C_K^{\alg}$ and $h \in C_K^{\alg}(y)^\times$. 
Then the generic type of \( y' = f(y) \) is orthogonal to the constants.
\end{theorem}

Explicit criteria are known for determining whether the differential equations   $\frac{\partial g}{\partial y} = f(y)$ and $\frac{\partial g}{\partial y}\Big/g=f(y)$ admit nonzero rational solutions by computing residues of the rational function $f(y)$. 
We now recall  the notion of the residue of a rational function.
Let $K$ be a  field of characteristic zero and $K(y)$ be the field of rational functions in $y$ over $K$. 
Let $f = \frac{a(y)}{b(y)} \in K(y)$ with $a, b \in K[y]$ and $\gcd(a, b) = 1$. Over the algebraic closure $K^{\alg}$,   $f$ admits a unique partial fraction decomposition
\[
f = p + \sum_{i=1}^{m} \sum_{j=1}^{n_i} \frac{\alpha_{i,j}}{(y - \beta_i)^j},
\]
where $p \in K[y]$, $m, n_i \in \mathbb{N}$, the $\beta_i\in K^{\alg}$  are
the distinct roots of $b$ with multiplicity $n_i$ and $\alpha_{i,j} \in K(\beta_1, \dots, \beta_m)$. The residue of $f$ at 
$\beta_i$ is defined by $\operatorname{res}(f, \beta_i) =\alpha_{i,1}$. In the case when $b$ is square-free, 
each $\beta_i$ is a simple pole and we have 
\[ \operatorname{res}(f, \beta_i) = \frac{a(\beta_i)}{b'(\beta_i)}. \]

\begin{prop} \label{prop-rationalsolresidue}
The following criteria hold.\begin{itemize}
    \item[(1)] 
(\cite[Proposition 2.1]{vanderput2001})
The differential equation $\frac{\partial g(y)}{\partial y} = f(y)$  has a nonzero rational  solution $g(y)\in K(y)$ if and only if all the residues of $f(y)$ are zero.

\item[(2)]
 (\cite[Proposition 3.1]{bostan2024}) The logarithmic differential equation
\(\frac{\partial g(y)/\partial y}{g(y)}=f(y)\)
admits a nonzero solution \(g(y)\in K(y)\)  if and only if  $f(y)$ has at most a simple pole with integral residue at each point of $K^{\text{alg}}$ and vanishes at infinity.
 \end{itemize} 
\end{prop}
 
\begin{lemma} \label{internal-squarefree}
Let $K$ be a field of constants and let $f\in K[y]$ be a square-free polynomial of degree $d\geq 3$. 
Let $r_1,\ldots,r_d$ be the residues of $1/f$ at the zeros of $f$.
Suppose there exists $i\neq j$ such that $r_i/r_j\notin\mathbb Q$. 
Then the generic type of $y'=f(y)$ is orthogonal to the constants.
\end{lemma}
\begin{proof}
By Theorem \ref{Rosenlicht},
it suffices to show that the differential equations $\frac{\partial g}{\partial y}=\frac{1}{f(y)}$ and $\frac{\partial h(y)}{\partial y}=\frac{c}{f}\cdot h(y)$ for all $c\in K^{\alg}\backslash\{0\}$
do not have nonzero solutions in $K^{\alg}(y)$. 
Assume the solutions of $f(y)=0$ are $\beta_1,\ldots, \beta_d$.
Since $f$ is square-free,  $\beta_i\neq\beta_j$ for $i\neq j$.
So $r_i:=\text{res}(1/f,\beta_i)=\frac{1}{f'(\beta_i)}$ and $\text{res}(c/f,\beta_i)=cr_i$.
 Since $r_i\neq0$, by Proposition \ref{prop-rationalsolresidue} (1),
 $\frac{\partial g}{\partial y}=\frac{1}{f(y)}$ does not have a rational function solution over $K$ (or equivalently over $K^{\alg}$).
 It remains to show there does not exist nonzero $c\in K^{\alg}$ such that $cr_i\in\mathbb Z$ for every $i=1,\ldots,d$. Suppose the contrary. Then for any pair $(i,j)$, $r_i/r_j\in\mathbb Q$, a contradiction to the hypothesis. Thus, by Proposition \ref{prop-rationalsolresidue} (2) and Theorem \ref{Rosenlicht}, the generic type of $y'=f(y)$ is orthogonal
to the constants.     
\end{proof}

\begin{lemma}\label{residue-irreduciblepoly}
Let $K$ be a field of constants 
of characteristic zero, and let
$f\in K[y]$ have degree $n\geq 3$.
Let $K_f$ be the subfield generated over $\mathbb Q$ by the
coefficients of $f$. If $f$ is irreducible over $K_f$, then the 
generic  type 
 of $\delta y=f(y)$ in $S_\delta(K)$ is
orthogonal to the constants.
\end{lemma}

 \begin{proof}
 Let $L\supset K_f$ be a splitting field of $f$, and write
\[
f(y)=\lambda\prod_{i=1}^n(y-\beta_i),
\qquad \lambda=\operatorname{lc}(f).
\]
Since $f$ is irreducible over $K_f$, its roots
are distinct and the residues of $1/f$ at these roots are
\[
r_i=\frac{1}{f'(\beta_i)}\in L^\times
\qquad (1\leq i\leq n),
\]
where $f'(y)=\frac{\partial f(y)}{\partial y}$.
 By Lemma \ref{internal-squarefree},
 it suffices to show that there exists $i\neq j$ such that $r_i/r_j\notin \mathbb Q$.
 
Suppose, towards a contradiction, that all these ratios lie in
$\mathbb Q$.
Since the Galois group $\operatorname{Aut}(L/K_f)$ acts transitively on the roots of $f$,
for each $i$, there exists $\tau_i\in \operatorname{Aut}(L/K_f)$ such that $\tau_i(\beta_1)=\beta_i$.
So \[r_i=\frac{1}{f'(\beta_i)}=\tau_i\big(\frac{1}{f'(\beta_1)}\big)=\tau_i(r_1)=q_{i1}r_1, \quad \text{for some } q_{i1}\in\mathbb Q.\]
Set $N:=|\operatorname{Aut}(L/K_f)|$. Then $r_1=\tau_i^N(r_1)=q_{i1}^Nr_1$ and $q_{i1}^N=1$.
As $q_{i1}\in\mathbb Q$, $q_{i1}=\pm 1$ and $r_i=\pm r_1$ for each $i$. 
Let $$I_1:=\{i:\,r_i=r_1\}, \qquad I_2:=\{i:\,r_i=-r_1\}.$$
So the partial fraction decomposition of $1/f$ over $L$ is $$\frac{1}{f}=\sum_{i\in I_1}\frac{r_1}{y-\beta_i}-\sum_{i\in I_2}\frac{r_1}{y-\beta_i}.$$
Let $P(y)=\prod_{i\in I_1}(y-\beta_i)\in L[y]$ and $Q(y)=\prod_{i\in I_2}(y-\beta_i)\in L[y]$, where an empty product is $1$. Then $P(y)$ and $Q(y)$ are coprime with $f=\lambda P(y)Q(y)$.
So we have $$\frac{1}{f}=r_1\frac{\partial P(y)/\partial y}{P(y)}-r_1\frac{\partial Q(y)/\partial y}{Q(y)}=r_1\frac{\lambda}{f}\cdot \left(\frac{\partial P(y)}{\partial y} Q-P\frac{\partial Q(y)}{\partial y}\right).$$
So $\frac{\partial P(y)}{\partial y} Q-P\frac{\partial Q(y)}{\partial y}=1/(r_1\lambda)\in L$.

If $I_2=\varnothing$, then $Q=1$ and $\deg P=n\geq 3$,
so $P'$ cannot be of degree zero. Thus this case is impossible.
Set $n_1=\deg P$ and $n_2=\deg Q$. If $n_1\neq n_2$,
then $P'Q-PQ'$ has leading coefficient $n_1-n_2\neq 0$ and degree
\[
n_1+n_2-1=n-1\geq 2,
\]
again a contradiction.
It remains to consider $n_1=n_2=k$. Then $n=2k\geq 3$, so $k\geq 2$.
Put $D=Q(y)-P(y)$. As $P$ and $Q$ are monic, $\deg(D)<k$ whenever $D\neq 0$.
If $D=0$, then $P'Q-PQ'=0$,  contradicting $P'Q-PQ'=1/(r_1\lambda)\neq0$. 
Otherwise, let $d=\deg D<k$. We have
$P'Q-PQ'=P'D-PD'.$
Its leading coefficient is $(k-d)\operatorname{lc}(D)\neq 0$,
and its degree is $k+d-1\geq 1$.
It therefore cannot be a nonzero element in $L$.
Thus, we get a contradiction. This proves the claim that there always exists $i\neq j$ such that $r_i/r_j\notin\mathbb Q$.  
By Lemma \ref{internal-squarefree},
the generic type is orthogonal to the constants.
\end{proof}
\vskip5pt

 In the special case where $f$ is a cubic polynomial irreducible over its field of definition, Lemma \ref{internal-squarefree} could be strengthened  by proving a stronger result in terms of the residues of $1/f$.
 \begin{lemma} \label{residue-cubic}
 Let $K$ be a $\delta$-field and let $f\in C_K[y]$ be a cubic polynomial irreducible over its field of definition $K_f$. 
 Let $r_1,r_2,r_3$ be the residues of $1/f$ at the zeros of $f$. 
 Then 
 \begin{itemize}
     \item the residues of $dy/f(y)$ at distinct poles are not rational
multiples of one another; equivalently, $r_i/r_j\notin\mathbb Q$ for all $1\leq i\neq j\leq 3$.
     \item The generic type of $y'=f(y)$ is  strictly minimal.
 \end{itemize}
 \end{lemma}
 \begin{proof}
 Scaling by the leading coefficient, we may assume that \(f\) is monic.
 Let $f=y^3+cy^2+by+d=(y-\alpha_1)(y-\alpha_2)(y-\alpha_3)$, where $\alpha_i\in K_f^{\text{alg}}$.
 Then $r_i:=\text{res}(1/f,\alpha_i)=1/f'(\alpha_i)$.
 Suppose on the contrary that there exists $i\neq j$ such that $r_i/r_j\in\mathbb Q$.
 Without loss of generality, suppose $r_1/r_2=-\frac{\alpha_2-\alpha_3}{\alpha_1-\alpha_3}=q\in\mathbb Q$.
 Since we also have $\alpha_1+\alpha_2+\alpha_3=-c$, $\alpha_2=q_1\alpha_3+c(q_1-1)/3$, where $q_1=(1+2q)/(1-q)\in\mathbb Q$ (here $q\neq1$, for otherwise, $\alpha_3=-c/3\in \mathbb Q(c)$).
 By taking the powers of $\alpha_2$ up to degree 3 and adding them together with corresponding coefficients of $f$, we get 
 \begin{eqnarray}
     0 &=&\alpha_2^3+c\alpha_2^2+b\alpha_2+d-q_1^3f(\alpha_3) \nonumber \\
 &=& q_1(q_1^2-1)(-b+c^2/3)\cdot\alpha_3\nonumber \\
 & & +[(q_1-1)^3c^3/27+(q_1-1)^2c^3/9+(q_1-1)bc/3+d-dq_1^3]. \nonumber
 \end{eqnarray}
Thus, $\alpha_3$ satisfies a linear polynomial, which must be 
the zero polynomial, by irreducibility of $f$. 
Note that $q_1\neq0$, or $1$, or $-1$, for otherwise, $\alpha_2=-c/3\in\mathbb Q(c)$, or $\alpha_2=\alpha_3$, or $\alpha_1=-c/3\in\mathbb Q(c)$ respectively, contradicting the irreducibility of $f$. So $b=c^2/3$.
Substituting $b=c^2/3$ into the degree zero term of the above equation, we get $(1-q_1^3)(d-c^3/27)=0$.
As $q_1\in\mathbb Q$ and $q_1\neq1$,  $d=c^3/27$.
So we have $f=y^3+cy^2+c^2y/3+c^3/27=(y+c/3)^3$, a contradiction to the irreducibility of $f$. Thus, no two distinct residues $r_i$, $r_j$ are rational multiples of each other. By Lemma \ref{hrushovski2003-lm2.23}, the generic type of $y'=f(y)$  is strictly minimal.
 \end{proof} 

We now consider the differential equation  $y'=y^3+a$ for $a\in C_{\mathbb{U}}\backslash\{0\}$ and  show that any transcendental solution to this differential equation does not satisfy any nonzero difference polynomial equation in a certain sense.

\begin{lemma}
\label{cor-pgl3}
Let $K$ be a $\sigma$-$\delta$-field of characteristic zero,
and let \(f(y)=y^3+a\) with $a\in K^\delta\backslash\{0\}$. Set
$f^\sigma(y):=y^3+\sigma(a)$, and consider
\[
V:=\Xi^*\left(\mathbb P^1,\frac{dy}{f(y)}\right),
\qquad
V^\sigma:=\Xi^*\left(\mathbb P^1,
                    \frac{dy}{f^\sigma(y)}\right).
\]
 Then 
\begin{enumerate}
    \item Both $V$ and $V^\sigma$ are strictly minimal
    over the constant field $k:=\mathbb Q(a,\sigma(a))$,
    relative to $\mathrm{DCF}_0$.
    \item $V$ and $V^\sigma$ are non-orthogonal if and
    only if $\sigma(a)=\pm a$.
\end{enumerate}
\end{lemma}

\begin{proof}
Choose $\beta\in\mathbb U$ with $\beta^3=-a$, and let
$\omega$ be a primitive cube root of unity.
Since $a\neq0$, 
\(\beta, \beta\omega\) and \(\beta\omega^2\)
are all the simple poles of $dy/f(y)$ and it has no other poles. Its residues are
\[
r_j
:=
\operatorname{res}\left(\frac{1}{f(y)}, \beta\omega^j\right)
=
\frac{1}{3\beta^2\omega^{2j}},
\quad j=0,1,2.
\]
For distinct $i,j\in\{0,1,2\}$, we have
\(
\frac{r_i}{r_j}
=
\omega^{2(j-i)}
\in\{\omega,\omega^2\}
\subseteq\mathbb U\setminus\mathbb Q.
\)
Thus no two residues are rational multiples of one another.
By Lemma~\ref{hrushovski2003-lm2.23}, $V$ is strictly
minimal over $k$.
Since \(
\sigma(a)\neq0\) and \(\delta(\sigma(a))=0,
\) the same argument applies to $V^\sigma$.

Applying Lemma~\ref{lemma-strictlyminimal} over $k^{\alg}$ to $V$ and $V^\sigma$, we have $V$ and $V^\sigma$ are non-orthogonal if and only if 
there exists $\varepsilon\in\{1,-1\}$ 
such that
\(
f^\sigma(\varepsilon y)
=
\varepsilon
f(y),
\)
equivalently,
\(
\sigma(a)=\varepsilon a.
\)
\end{proof}

\begin{theorem}
Let \(K\) be an algebraically closed
\(\sigma\)-\(\delta\)-field of characteristic zero, and let
\(a\in K^\delta\smallsetminus\{0\}\).
Let \(L\) be a \(\sigma\)-\(\delta\)-field extending \(K\),
and let \(\xi\) be a solution of
\(
\delta y=y^3+a
\)
that is transcendental over \(L\).
Then exactly one of the following holds:
\begin{enumerate}
    \item \(\xi\) is \(\sigma\)-transcendental over \(L\);
    \item there exist \(n\in\mathbb N_{>0}\) and
    \(\varepsilon\in\{1,-1\}\) such that
    \[
    \sigma^n(\xi)=\varepsilon\xi.
    \]
    In this case, necessarily
    \(\sigma^n(a)=\varepsilon a\).
\end{enumerate}
\end{theorem}

\begin{proof}
Set \(M=L^{\mathrm{alg}}\subset\mathbb U\). Then \(M\) is an algebraically closed \(\sigma\)-\(\delta\)-field, and \(\xi\) remains transcendental over \(M\). For each \(i\geq 0\), let
\[
W_i=\mathbb V\bigl(\delta y-y^3-\sigma^i(a)\bigr),
\]
and let \(p_i\in S_{\delta}(K)\) be its generic type over \(K\).
By Lemma~\ref{cor-pgl3}, each \(p_i\) is strictly minimal relative to $\operatorname{DCF}_0$. Hence, by
Lemma~\ref{lm-strictminimal} and the remark following it, each \(p_i\) is totally disintegrated.

Since \(\xi\) is transcendental over \(M\),
\(\operatorname{tp}_\delta(\xi/M)\) is a  nonforking extension of $p_0$ over $M$. Minimality of \(p_0\) therefore gives
\(\xi\ind_K^\delta M.
\)
Applying Theorem~\ref{thm-trivialminimal} over  \(K=K^{\alg}\subseteq M\), we obtain
\(
\xi\ind_K^{\sigma\text{-}\delta}M.
\)
By Proposition~\ref{prop-nonforking-ideal},
\(I_{\sigma\text{-}\delta}(\xi/M)
= I_{\sigma\text{-}\delta}(\xi/K).
\)
In particular, for every \(r\geq 0\), the entries of \(
\bigl(\xi,\sigma(\xi),\ldots,\sigma^r(\xi)\bigr)\)
are algebraically independent over \(M\) if and only if they are algebraically independent over \(K\).
Thus \(\xi\) is \(\sigma\)-transcendental over \(M\) if and
only if it is \(\sigma\)-transcendental over \(K\).

Suppose that \(\xi\) is not \(\sigma\)-transcendental over
\(L\). Then it is \(\sigma\)-algebraic over \(M\), and hence, is \(\sigma\)-algebraic over \(K\).
Apply Theorem~\ref{thm-strictmin} to \(f(y)=y^3+a\).
There exist \(n\geq 1\), \(\varepsilon\in\{1,-1\}\), and
\(b\in K^\delta\) such that
\[
\sigma^n(\xi)=\varepsilon\xi+b
\]
and
\[
y^3+\sigma^n(a)
=
\varepsilon f\!\left(\frac{y-b}{\varepsilon}\right)
=
(y-b)^3+\varepsilon a.
\]
Comparing coefficients of \(y^2\) and the constant terms on both sides  gives \(b=0\) and
\(\sigma^n(a)=\varepsilon a\).
Consequently, 
\(
\sigma^n(\xi)=\varepsilon\xi,
\)
as required.
\end{proof}

More generally, we have the following result:

\begin{cor}\label{cor:irreducible-cubic-difference}
Let $K$ be an algebraically closed $\sigma$-$\delta$-field of
characteristic zero, and let
\[
f(y)=y^3+a_2y^2+a_1y+a_0\in C_K[y]
\]
be irreducible over $\mathbb Q(a_0,a_1,a_2)$.
Let $L\supseteq K$ be any $\sigma$-$\delta$-field  and let $\xi$
be a solution of $\delta y=f(y)$ with $\xi\notin L^{\mathrm{alg}}$.
Then either $\xi$ is $\sigma$-transcendental over $L$, or there exist
$s\geq 1$, $\varepsilon\in\{1,-1\}$, and $b=\frac{\varepsilon a_2-\sigma^s(a_2)}{3} \in C_K$ such that
\[
\sigma^s(\xi)=\varepsilon\xi+b.
\]
In the second case, necessarily,
$f^{\sigma^s}(y)=\varepsilon f\bigl(\varepsilon(y-b)\bigr)$.
\end{cor}

\begin{proof}
Set $M=L^{\mathrm{alg}}\subset\mathbb U$. Then $M$ is an algebraically closed $\sigma$-$\delta$-field
containing $K$, and $\xi$ is transcendental over $M$.
By Lemma~\ref{residue-cubic}, the nonconstant solution locus
of $\delta y=f(y)$ has a strictly minimal generic type.
Its generic type over $M$ is a nonforking extension and is again
totally disintegrated by  Corollary~\ref{cor:trivial-field-of-definition-of-dependence}.

Suppose that $\xi$ is $\sigma$-algebraic over $L$.
It is then $\sigma$-algebraic over $M$, so
Theorem~\ref{thm-strictmin}, applied over $M$, yields
$s\geq 1$, $\varepsilon\in\{1,-1\}$, and $b\in C_M$ such that
\[
\sigma^s(\xi)=\varepsilon\xi+b,
\qquad
f^{\sigma^s}(y)
=\varepsilon f\bigl(\varepsilon(y-b)\bigr).
\]
Comparing coefficients of $y^2$ gives
\[
\sigma^s(a_2)=\varepsilon a_2-3b.
\]
Hence $b=(\varepsilon a_2-\sigma^s(a_2))/3\in C_K$. \end{proof}

\section{Counterpoint for non-trivial types}
\label{sec:non-trivial-types}

In this section we examine some examples of non-trivial 
minimal types relative to \(\DCF\) and consider their
minimality relative to \(\DCFA\).  We refer 
to Lascar rank, or what is often called 
\(\operatorname{SU}\)-rank for simple theories.
See \cite[Chapter 5]{Wagner-simple} for details.
For later use, we recall its characterization in $\DCFA$.

\begin{remark}
\label{rk:SU-rank-DCFA}
Fix a completion $T$ of $\DCFA$ and a monster model
$\mathbb U$ of $T$.
By \cite[Proposition~3.1(v)]{LS-DCFAm}, $T$ is
supersimple, so every complete type in finitely many
variables has ordinal-valued $\operatorname{SU}$-rank
\cite[Proposition~6.9(ii)]{KimPillay1997}.
In this section, $\operatorname{SU}(a/K)$ denotes the rank of the
complete first-order type $\operatorname{tp}(a/K)$
relative to $\DCFA$.

For a small $\sigma$-$\Delta$-subfield $K$ of $\mathbb U$
and a finite tuple $a$, write
\(
K\langle a\rangle_{\sigma\text{-}\Delta}
=
K\bigl\langle\sigma^j(a)
:\ j\in\mathbb Z
\bigr\rangle_{\Delta}.
\)
By \cite[Proposition~3.1(v)]{LS-DCFAm}, for every small
$\sigma$-$\Delta$-field $L$ extending $K$,
\[
a\ind_K^{\sigma\text{-}\Delta}L
\quad\Longleftrightarrow\quad
K\langle a\rangle_{\sigma\text{-}\Delta}
\text{ and }L
\text{ are algebraically disjoint over }K.
\]
Consequently, $\operatorname{SU}(a/K)\geq0$, and
\[
\operatorname{SU}(a/K)\geq\alpha+1
\]
if and only if there is a small
$\sigma$-$\Delta$-field $L\supseteq K$ such that
$\operatorname{SU}(a/L)\geq\alpha$ and
$K\langle a\rangle_{\sigma\text{-}\Delta}$ and $L$
are not algebraically disjoint over $K$.
For a nonzero limit ordinal $\lambda$,
$\operatorname{SU}(a/K)\geq\lambda$ if and only if
$\operatorname{SU}(a/K)\geq\alpha$ for every
$\alpha<\lambda$.
\end{remark}

We start with some general results.  To state the
first of these we need the definition of 
the \(\sigma\)-order.  For the 
remainder of this section, \(K \subseteq \mathbb{U}\)
is a small inversive \(\sigma\)-\(\Delta\)-subfield of 
\(\mathbb{U}\) and \(a \in \mathbb{U}^x\) is a 
tuple from \(\mathbb{U}\).

\begin{definition}
\label{def:sigma-degree}
We define the \(\sigma\)-order 
of \(a\) over \(K\), \(\operatorname{ord}_\sigma(a/K)\),
to be 
\begin{itemize}
    \item \(0\) if \(a\) is algebraic over \(K\),
    \item \(\ell \in \mathbb{Z}_+\) if \(\ell\) is 
    minimal with \(K \langle a \rangle_{\sigma\text{-}\Delta} 
\subseteq K \langle a, \ldots, \sigma^{\ell-1} a\rangle_\Delta^{\operatorname{alg}}\), and
\item \(\infty\) otherwise.
\end{itemize}
For a partial type \(\Sigma\) defined over 
\(K\) we define \(\ord_\sigma(\Sigma) := \sup \{ \ord_\sigma (a/K) : a \in \Sigma(\mathbb{U}) \} \). 
\end{definition}

\begin{lemma}
\label{lem:fork-minimal-finite}
If \(\operatorname{tp}_\Delta(a/K)\) is 
minimal and \(L \subseteq \mathbb{U}\) is 
a small inversive \(\sigma\)-\(\Delta\)-extension of \(K\)
with \(a \nind^{\sigma\text{-}\Delta}_K L\), then 
\(\ord_\sigma(a/L) < \infty\).
\end{lemma}
\begin{proof}
By definition of \(\ind^{\sigma\text{-}\Delta}\), we have that 
\( (\sigma^j(a))_{j \in \mathbb{Z}} \nind^\Delta_K L\).  
By finite character of dependence, there are \(M,N \in \mathbb{N}\)
so that \((\sigma^j(a))_{-M \leq j \leq N} \nind_K^\Delta L\).  Applying
\(\sigma^M\) and using the fact that \(K\) and \(L\) are 
inversive, we may assume that \(M = 0\).  Minimize  \(N\). If \(N = 0\), then \(a \in L^{\operatorname{alg}}\) which implies because \(L\) is inversive that \(\sigma^i(a)\) is algebraic over \(L\) for every \(i \in \mathbb{Z}\). That is, \(\ord_\sigma(a/L) = 0\).  Consider now the case that \(N > 0\).
We have that \(a, \ldots, \sigma^{N-1}(a) \ind^\Delta_K L\)  so 
that by transitivity \(\sigma^N(a) \nind^\Delta_{K \langle a, \ldots, \sigma^{N-1}(a) \rangle_\Delta} L \langle a, \ldots, \sigma^{N-1}(a) \rangle_\Delta \).  Since \(\operatorname{tp}_\Delta(a/K)\) 
is minimal, so is \(\operatorname{tp}_\Delta(\sigma^N(a)/K)\).
Hence, this dependence gives that 
\(\sigma^N(a)\) is algebraic over 
\(L \langle a, \ldots, \sigma^{N-1}(a) \rangle_\Delta\).  
Iterating \(\sigma\), we see that \[L \langle \{ \sigma^j(a) 
: j \in \mathbb{N} \} \rangle_\Delta 
\subseteq L \langle a, \ldots, \sigma^{N-1}(a) \rangle_\Delta^{\operatorname{alg}}.\]

 Applying $\sigma^{-1}$ to the above gives $\sigma^{N-1}(a)\in L\langle\sigma^{-1}(a),a,\ldots,\sigma^{N-2}(a)\rangle_{\Delta}^{\alg}$. We claim that $\sigma^{N-1}(a)\notin L\langle a,\ldots,\sigma^{N-2}(a)\rangle_{\Delta}^{\alg}$.
 Otherwise, since $a,\ldots,\sigma^{N-1}(a)\ind^\Delta_K L$, we would get $\sigma^{N-1}(a)\in K\langle a,\ldots,\sigma^{N-2}(a)\rangle_{\Delta}^{\alg}$.  Applying $\sigma$, we get  $\sigma^{N}(a)\in K\langle \sigma(a),\ldots,\sigma^{N-1}(a)\rangle_{\Delta}^{\alg}$, whence $ K\langle a,\ldots,\sigma^N(a)\rangle_{\Delta}^{\alg}=K\langle a,\ldots,\sigma^{N-1}(a)\rangle_{\Delta}^{\alg}$ would be independent from $L$ over $K$, a contradiction to $a,\ldots,\sigma^{N}(a)\nind^\Delta_K L$. This proves the claim.
 Thus, we have $\sigma^{N-1}(a)\nind^\Delta_{L\langle a,\ldots,\sigma^{N-2}(a)\rangle_\Delta} \sigma^{-1}(a)$.
 By symmetry and minimality of $\operatorname{tp}_\Delta(\sigma^{-1}(a)/K)$, we obtain 
 $\sigma^{-1}(a)\in L\langle a,\ldots,\sigma^{N-1}(a)\rangle_{\Delta}^{\alg}$.
Iterating \(\sigma^{-1}\),  we see that \(L \langle a \rangle_{\sigma\text{-}\Delta} = L \langle \{ \sigma^j(a) 
: j \in \mathbb{Z} \} \rangle_\Delta 
\subseteq L \langle a, \ldots, \sigma^{N-1}(a) \rangle_\Delta^{\operatorname{alg}}\). 
That is, \(\ord_\sigma(a/L) \leq N\).  
\end{proof}

\begin{prop}
\label{prop:minimal-sigma-degree}
If \(\operatorname{tp}_\Delta(a/K)\) is minimal, 
then \(\operatorname{SU}(a/K) \leq \ord_\sigma(a/K)\).
\end{prop}

Note that in Proposition~\ref{prop:minimal-sigma-degree}, if 
\(\ord_\sigma(a/K) = \infty\), then the inequality 
\(\operatorname{SU}(a/K) \leq \ord_\sigma(a/K)\) holds 
trivially:  by convention \(\infty\) is greater than 
every ordinal.

\begin{proof}
Taking into account the observation made just before this proof, 
it suffices to show that for natural numbers \(N\)
if \(\ord_\sigma(a/K) \leq N\), then 
\(\operatorname{SU}(a/K) \leq N\). We do this 
by induction on \(N\).  The case of \(N = 0\)
is trivial as \(\ord_\sigma(a/K) \leq 0\) implies by 
definition that \(a\) is model-theoretically algebraic over
\(K\) relative to \(\DCFA\) while the hypothesis on \(K\) 
that it is an inversive \(\sigma\)-\(\Delta\)-field
implies that this closure is the ordinary algebraic 
closure which would then imply that even relative to 
\(\DCF\) the type \(\operatorname{tp}_\Delta(a/K)\) is 
algebraic, and, hence, not minimal.
For \(N+1\), if \(\operatorname{SU}(a/K) \geq N+2\),
then by definition of the Lascar rank we can find a 
small inversive \(\sigma\)-\(\Delta\)-extension 
\(L \subseteq \mathbb{U}\) of \(K\) so that 
\(a \nind_K^{\sigma\text{-}\Delta} L\) and 
\(\operatorname{SU}(a/L) \geq N+1\).  By hypothesis, 
\(\ord_\sigma(a/K) \leq N+1\).  
That is \(K \langle a \rangle_{\sigma\text{-}\Delta} 
\subseteq K \langle a, \ldots, \sigma^N(a) \rangle_\Delta^{\operatorname{alg}}\). 
From the definition of \(\ind^{\sigma\text{-}\Delta}\), we would have 
\((\sigma^j(a))_{j\in \mathbb{Z}} \nind_K^\Delta L\).  
As the full sequence is algebraic over \((\sigma^j(a))_{j=0}^{N}\), 
we would have \(a, \ldots, \sigma^{N} a \nind_K^\Delta L\).  
Taking \(M \leq N\) minimal with \(a, \ldots, \sigma^M a \nind_K^\Delta L\), we would have by 
transitivity that \(\sigma^M a \nind_{K \langle a, \ldots, 
\sigma^{M-1}(a) \rangle_\Delta} L \langle a, \ldots, 
\sigma^{M-1}(a) \rangle_\Delta \).  As 
\(\operatorname{tp}_\Delta(a/K)\), and hence also \(\operatorname{tp}_\Delta(\sigma^M(a)/K)\), 
is minimal, this dependence yields that \(\sigma^M(a)\) is 
algebraic over \(L \langle a, \ldots, \sigma^{M-1}(a) \rangle_\Delta\).  Iterating \(\sigma\) and \(\sigma^{-1}\) as in the proof of Lemma \ref{lem:fork-minimal-finite}, this 
yields that \(L \langle a \rangle_{\sigma\text{-}\Delta} 
\subseteq L \langle a, \ldots, \sigma^{M-1}(a) \rangle_\Delta^{\operatorname{alg}} \subseteq  L \langle a, \ldots, \sigma^{N-1}(a) \rangle_\Delta^{\operatorname{alg}}\).  
That is, \(\ord_\sigma(a/L) \leq N\).  Hence, 
by induction, \(\operatorname{SU}(a/L) \leq N\), 
contrary to our hypothesis.  With this we conclude.
\end{proof}

\begin{cor}
\label{cor:at-most-oemga}
If \(\operatorname{tp}_\Delta(a/K)\) is minimal, then 
\(\operatorname{SU}(a/K) \leq \omega\).
\end{cor}
\begin{proof}
Suppose to the contrary that \(\operatorname{SU}(a/K) > \omega\).
Then by definition of the Lascar rank
we can find a small inversive \(\sigma\)-\(\Delta\)-extension 
\(L \subseteq \mathbb{U}\) of \(K\) so that 
\(a \nind^{\sigma\text{-}\Delta}_K L\) and \(\operatorname{SU}(a/L) 
\geq \omega\).  By Lemma~\ref{lem:fork-minimal-finite}, 
\(\ord_\sigma(a/L) < \infty\).  But then 
Proposition~\ref{prop:minimal-sigma-degree} yields
\(\operatorname{SU}(a/L) \leq \ord_\sigma(a/L) < \omega\), 
contrary to our hypothesis.
\end{proof}

\begin{prop}
\label{prop:groups-as-infinite-rank}
Let \(G\) be a definable, connected, minimal group relative 
to \(\DCF\).  Suppose that for some \(\ell \in \mathbb{Z}_+\)
we have a definable surjective homomorphism \(f:G \to G^{\sigma^\ell}\).  Then relative to \(\DCFA\) we have
\(\operatorname{SU}(G) \geq \omega\).
\end{prop}

\begin{proof}
Since a minimal group is commutative~\cite{Reineke}, 
we use additive notation.
Define the $\DCFA$-definable endomorphism
\[
\psi:G\longrightarrow G,
\qquad
\psi(x):=x-\sigma^{-\ell}(f(x)).
\]
For each $N\geq0$, let $H_N:=\ker(\psi^N)$, where
$\psi^0=\operatorname{id}_G$ and $H_0=\{0\}$.
Thus,
\(
H_1
=\{x\in G:\sigma^\ell(x)=f(x)\}.
\)

Let us now show how to use the geometric axioms to 
see that \(\psi\) is surjective and that 
\(H_1\) is infinite.

By \cite[Proposition~3.4(v)]{LS-DCFAm},
$(\mathbb U,\Delta,\sigma^\ell)$ is a model of $\DCFA$.
 Fix $b\in G(\mathbb U)$ and an arbitrary finite set
$E\subseteq G(\mathbb U)$.
Choose a small algebraically closed
$\sigma^\ell$-$\Delta$-subfield $K\subseteq\mathbb U$
containing $b$, the coordinates of $E$,
and the parameters defining $G$ and $f$.
Since $G$ is connected, it has a unique generic
differential type $p$ over $K$.
By saturation, choose $c\in G(\mathbb U)$ realizing $p$,
and set $d:=f(c)+\sigma^\ell(b)$.
Then $d$ realizes the unique generic differential
type $p^{\sigma^\ell}$ of $G^{\sigma^\ell}$ over $K$.

Let $Z$ be the Kolchin locus of $c$ over $K$, and let
$W$ be the Kolchin locus of $(c,d)$ over $K$.
Since $d$ realizes $p^{\sigma^\ell}$, its Kolchin locus is
$Z^{\sigma^\ell}$. Hence
\(
W\subseteq Z\times Z^{\sigma^\ell},
\)
and both coordinate projections are $\Delta$-dominant.
Consider the $K$-definable set
\[
D:=\{(x,y)\in G\times G^{\sigma^\ell}:
      x\notin E,\ y=f(x)+\sigma^\ell(b)\}.
\]
Note that $(c,d)\in D$, since $c$ is generic
over $K$ and $E\subseteq G(K)$ is finite.
By quantifier elimination in $\DCF$, $D$ is
Kolchin-constructible. Because it contains the generic
point $(c,d)$ of $W$, there is a nonempty
$\Delta$-open subset $O\subseteq W$ contained in $D$.

Applying \cite[Fact~2.1]{LS-DCFAm} to
$W\subseteq Z\times Z^{\sigma^\ell}$ and its open subset $O$,
we obtain $x\in Z(\mathbb U)$ such that
\(
(x,\sigma^\ell(x))\in O\subseteq D.
\)
Therefore $x\in G(\mathbb U)\setminus E$ and
\(
\sigma^\ell(x)=f(x)+\sigma^\ell(b).
\)
Applying $\sigma^{-\ell}$ gives
\(
\psi(x)=x-\sigma^{-\ell}(f(x))=b.
\)

Taking $E=\varnothing$ and allowing arbitrary $b$
proves that $\psi$ is surjective.
Taking $b=0$ and allowing arbitrary finite $E$
shows that $H_1=\ker\psi$ is infinite.

For every $N\geq0$, the restriction
$\psi^N:H_{N+1}\longrightarrow H_1$
is a surjective homomorphism with kernel $H_N$.
Indeed, if $x\in H_{N+1}$, then
$\psi(\psi^N(x))=0$, so $\psi^N(x)\in H_1$.
Conversely, given $z\in H_1$, surjectivity of $\psi^N$
provides $x\in G$ with $\psi^N(x)=z$, and then
$\psi^{N+1}(x)=\psi(z)=0$.
The kernel of this restriction is precisely $H_N$.
Therefore,
\(
H_{N+1}/H_N\cong H_1.
\)
In particular, $[H_{N+1}:H_N]$ is infinite.

Thus $G$ contains an increasing chain of
$\DCFA$-definable subgroups
\(
\{0\}=H_0<H_1<H_2<\cdots,
\)
each of infinite index in the next. 
By the Lascar inequalities
\cite[Theorem~5.1.6(1)]{Wagner-simple},
\(
\operatorname{SU}(H_{N+1})
\geq
\operatorname{SU}(H_N)
+\operatorname{SU}(H_{N+1}/H_N)
\geq
\operatorname{SU}(H_N)+1.\)
It follows that $\operatorname{SU}(G)\geq\omega$.
\end{proof}

For a complete $\Delta$-type $p \in S_{x,\Delta}(K)$ over $K$, 
viewed as a partial
type in the language of $\DCFA$, write
\[
\operatorname{SU}(p)
:=
\sup\bigl\{
\operatorname{SU}(q):
q\in S_{x,\sigma\text{-}\Delta}(K),\ p\subseteq q
\bigr\}.
\]
Here, $S_{x,\sigma\text{-}\Delta}(K)$ denotes the space of complete
types over $K$ in the full $\sigma$-$\Delta$-language replacing our earlier convention to consider only quantifier-free types.

\begin{cor}
\label{cor:rank-omega-type}
Let $K$ be a small $\sigma$-$\Delta$-field.
Suppose that $p=\operatorname{tp}_\Delta(a/K)$
is minimal and is the generic type of a connected
$K$-definable group $G$ relative to $\DCF$.
Suppose further that there is a $\DCF$-definable
surjective homomorphism
\(
f:G\longrightarrow G^{\sigma^\ell}
\)
for some $\ell\in\mathbb Z_{>0}$.
Then $\operatorname{SU}(p)=\omega$ relative to $\DCFA$.
\end{cor}

\begin{proof}
Since $G$ is connected and the 
$\Delta$-type $p$ is minimal, $G$ is strongly minimal relative
to $\DCF$. Thus $p$ is the unique nonalgebraic
$\Delta$-type over $K$  on $G$.

Let $q$ be a generic type of $G$ over $K$ relative
to $\DCFA$. Its $\Delta$-restriction is nonalgebraic:
otherwise, $q$ would contain a $\Delta$-formula
over $K$ with finitely many solutions and would
itself be algebraic. Consequently, $q$ extends $p$.

Generic types of a definable group in a supersimple
theory have the same $\operatorname{SU}$-rank as
the group. Hence Proposition~\ref{prop:groups-as-infinite-rank}
gives
\[
\operatorname{SU}(p)
\geq \operatorname{SU}(q)
=\operatorname{SU}(G)
\geq\omega.
\]
On the other hand, let $q$ be any complete
$\sigma$-$\Delta$-type over $K$ extending $p$,
and  $b\models q$. Since
$\operatorname{tp}_\Delta(b/K)=p$ is minimal,
Corollary~\ref{cor:at-most-oemga} gives
$\operatorname{SU}(q)=\operatorname{SU}(b/K)\leq\omega$.
Taking the supremum over all such completions $q$,
we obtain $\operatorname{SU}(p)\leq\omega$.
Therefore $\operatorname{SU}(p)=\omega$.
\end{proof}

\begin{remark}
The equality $\operatorname{SU}(p)=\omega$ concerns
$p$ as a partial $\sigma$-$\Delta$-type.
It does not assert that every complete
$\sigma$-$\Delta$-type extending $p$ has rank $\omega$.
In particular, it does not necessarily imply
$\operatorname{SU}(\operatorname{tp}_{\sigma\text{-}\Delta}(a/K))=\omega$
for the particular realization $a$ used to specify $p$.
However, the proof does show that some completion of $p$
has rank $\omega$.
\end{remark}

\begin{cor}
\label{cor:generic-constant-as-sigma}
Let \(p \in S_\Delta(K)\) be the generic type of the 
field of constants.  That is \(p\) is axiomatized by 
\(\delta x = 0\) for \(\delta \in \Delta\) and 
\(g(x)\neq 0\) for every nonzero $g\in K[x]$.  Then \(p\) is
minimal relative to \(\DCF\)  and \(\operatorname{SU}(p) = \omega\) 
relative to \(\DCFA\). 
\end{cor}

\begin{proof}
Apply Corollary~\ref{cor:rank-omega-type} to the 
group \(G = (\{ x \in \mathbb{U} : \delta x = 0 
\text{ for all } \delta \in \Delta \}, +)\) and the definable homomorphism 
\(f:G \to G\) being the identity map given by \(x \mapsto x\).
\end{proof}

For the next proposition, the reader should  
consult~\cite{Marker-manin-kernels} for the theory of 
Manin kernels.  As this theory is worked out in the 
literature only in the case of \(m=1\), we restrict to 
$\Delta=\{\delta\}$.

\begin{theorem}
 \label{prop:Manin-kernel-isogenous}
Let $K$ be an algebraically closed $\sigma$-$\delta$-field.
Suppose that $A$ is a simple
abelian variety over $K$ that does not descend to
the constants; that is, $A$ is not isomorphic over
$\mathbb U$ to an abelian variety defined over
$\mathsf C$.
Let $p$ be the generic differential type of
the Manin kernel $A^\sharp$ over $K$.

Then  relative to $\DCFA$, either $p$ remains minimal
or $\operatorname{SU}(p)=\omega$.
The first case occurs if and only if there is no
positive integer $\ell$ and isogeny
\(f:A\longrightarrow A^{\sigma^\ell}\).
 In that case, $p$ has a unique complete quantifier-free
$\sigma$-$\delta$-extension over $K$, this extension
is minimal, and every full completion of $p$ has
$\operatorname{SU}$-rank one.
\end{theorem}

\begin{proof}
By \cite[Fact~5.1]{Casale-Freitag-Nagloo-I},
$A^\sharp$ is strongly minimal.
In particular, $p$ is minimal and stationary.
For each $\ell>0$, the generic $\Delta$-type of
$(A^{\sigma^\ell})^\sharp=
(A^\sharp)^{\sigma^\ell}$
is $p^{\sigma^\ell}$.
By \cite[Fact~5.6]{Casale-Freitag-Nagloo-I},
for every $\ell>0$,
\[
p\not\perp p^{\sigma^\ell}
\quad\Longleftrightarrow\quad
A\text{ and }A^{\sigma^\ell}\text{ are isogenous}.
\]

Suppose first that no such isogeny exists. 
Then $p\perp p^{\sigma^\ell}$ for every $\ell>0$.
By Proposition~\ref{prop:automatic-sigma-hypertranscendence}, $p$ has a unique complete quantifier-free
$\sigma$-$\delta$-extension, and this extension
is minimal.  It remains to show that every full completion of $p$
has $\operatorname{SU}$-rank one.
Let $q$ be such a completion, let $a\models q$, and
let $E\supseteq K$ be a small parameter set such that
$\operatorname{tp}_{\sigma\text{-}\delta}(a/E)$
is nonalgebraic. Set
\(
L:=\operatorname{acl}^{\sigma\text{-}\delta}(E).
\)
By \cite[Proposition~3.1(i)]{LS-DCFAm}, $L$ is an
algebraically closed $\sigma$-$\delta$-field.
Moreover, $a$ is nonalgebraic over $L$.
The proof of
Proposition~\ref{prop:automatic-sigma-hypertranscendence}
therefore gives
\(
a\ind^{\sigma\text{-}\delta}_K L,
\)
and hence $a\ind^{\sigma\text{-}\delta}_K E$.
Thus every nonalgebraic extension of $q$ is
nonforking, so every forking extension of $q$ is
algebraic. Since $q$ itself is nonalgebraic,
$\operatorname{SU}(q)=1$.

Suppose now that there is an isogeny
$f:A\to A^{\sigma^\ell}$ for some $\ell>0$, which necessarily will be defined over \(K\)~\cite{Chow-AV}. 
We show that $f$ restricts to a definable surjective
homomorphism $A^\sharp\to (A^{\sigma^\ell})^\sharp$.

Since   $f$ sends $\operatorname{Tor}(A)$, the set of torsion points of $A$, to $\operatorname{Tor}(A^{\sigma^\ell})$,  
$\operatorname{Tor}(A)\subseteq f^{-1}((A^{\sigma^\ell})^\sharp)$.
As the Manin kernel $A^\sharp$ is the  smallest $\Delta$-definable subgroup containing $\operatorname{Tor}(A)$\cite[Lemma~3.3(i)]{BBP-semiabelian}, \(A^\sharp\subseteq f^{-1}((A^{\sigma^\ell})^\sharp)\),  and thus
\(f(A^\sharp)\subseteq (A^{\sigma^\ell})^\sharp.
\)
 Since $f$ has finite kernel and $A^\sharp$ is
infinite, $f(A^\sharp)$ is an infinite $\Delta$-definable subgroup of $(A^{\sigma^\ell})^\sharp$.
Strong minimality of $(A^{\sigma^\ell})^\sharp$ therefore gives $f(A^\sharp)=(A^{\sigma^\ell})^\sharp$.

We have therefore obtained a $K$-definable
surjective homomorphism
\[
f|_{A^\sharp}:
A^\sharp\longrightarrow (A^{\sigma^\ell})^\sharp=
 (A^\sharp)^{\sigma^\ell}.
\]
The strongly minimal group $A^\sharp$ is connected:
a proper definable subgroup of finite index would
be infinite and coinfinite.
Corollary~\ref{cor:rank-omega-type} now gives
$\operatorname{SU}(p)=\omega$.

\end{proof}

\bibliographystyle{siam}
\bibliography{ode-as-diff-eq}

@article {ADH-JAMS,
    AUTHOR = {Adamczewski, Boris and Dreyfus, Thomas and Hardouin,
              Charlotte},
     TITLE = {Hypertranscendence and linear difference equations},
   JOURNAL = {J. Amer. Math. Soc.},
  FJOURNAL = {Journal of the American Mathematical Society},
    VOLUME = {34},
      YEAR = {2021},
    NUMBER = {2},
     PAGES = {475--503},
      ISSN = {0894-0347,1088-6834},
   MRCLASS = {12H05 (11J81 39A06)},
  MRNUMBER = {4280865},
MRREVIEWER = {Russell\ Jay\ Hendel},
       DOI = {10.1090/jams/960},
       URL = {https://doi.org/10.1090/jams/960},
}

@article {Chow-AV,
    AUTHOR = {Chow, Wei-Liang},
     TITLE = {Abelian varieties over function fields},
   JOURNAL = {Trans. Amer. Math. Soc.},
  FJOURNAL = {Transactions of the American Mathematical Society},
    VOLUME = {78},
      YEAR = {1955},
     PAGES = {253--275},
      ISSN = {0002-9947,1088-6850},
   MRCLASS = {14.0X},
  MRNUMBER = {67527},
MRREVIEWER = {P.\ Samuel},
       DOI = {10.2307/1993062},
       URL = {https://doi.org/10.2307/1993062},
}

@article{Nagloo–Pillay-2017,
  title={On algebraic relations between solutions of a generic {P}ainlev{\'{e}} equation},
  author={Nagloo, Joel and Pillay, Anand},
  journal={Journal f{\"u}r die reine und angewandte Mathematik (Crelles Journal)},
  volume={726},
  pages={1--27},
  year={2017},
  publisher={De Gruyter}
}

@article {Pillay-foundational,
    AUTHOR = {Pillay, Anand},
     TITLE = {Some foundational questions concerning differential algebraic
              groups},
   JOURNAL = {Pacific J. Math.},
  FJOURNAL = {Pacific Journal of Mathematics},
    VOLUME = {179},
      YEAR = {1997},
    NUMBER = {1},
     PAGES = {179--200},
      ISSN = {0030-8730,1945-5844},
   MRCLASS = {12H05},
  MRNUMBER = {1452531},
MRREVIEWER = {Alexandru\ Buium},
       DOI = {10.2140/pjm.1997.179.179},
       URL = {https://doi.org/10.2140/pjm.1997.179.179},
}

@article {Reineke,
    AUTHOR = {Reineke, Joachim},
     TITLE = {Minimale {G}ruppen},
   JOURNAL = {Z. Math. Logik Grundlagen Math.},
  FJOURNAL = {Zeitschrift f\"ur Mathematische Logik und Grundlagen der
              Mathematik},
    VOLUME = {21},
      YEAR = {1975},
    NUMBER = {4},
     PAGES = {357--359},
      ISSN = {0044-3050},
   MRCLASS = {02H15 (02G20 20K99)},
  MRNUMBER = {379179},
MRREVIEWER = {U.\ Felgner},
       DOI = {10.1002/malq.19750210145},
       URL = {https://doi.org/10.1002/malq.19750210145},
}

@article {Sh-dcf,
    AUTHOR = {Shelah, Saharon},
     TITLE = {Differentially closed fields},
   JOURNAL = {Israel J. Math.},
  FJOURNAL = {Israel Journal of Mathematics},
    VOLUME = {16},
      YEAR = {1973},
     PAGES = {314--328},
      ISSN = {0021-2172},
   MRCLASS = {02H15 (12H05)},
  MRNUMBER = {344116},
MRREVIEWER = {Carol\ Wood},
       DOI = {10.1007/BF02756711},
       URL = {https://doi.org/10.1007/BF02756711},
}

@article {Kolchin-constrained,
    AUTHOR = {Kolchin, E. R.},
     TITLE = {Constrained extensions of differential fields},
   JOURNAL = {Advances in Math.},
  FJOURNAL = {Advances in Mathematics},
    VOLUME = {12},
      YEAR = {1974},
     PAGES = {141--170},
      ISSN = {0001-8708},
   MRCLASS = {12H05 (02H15)},
  MRNUMBER = {340227},
MRREVIEWER = {Jerald\ Kovacic},
       DOI = {10.1016/S0001-8708(74)80001-0},
       URL = {https://doi.org/10.1016/S0001-8708(74)80001-0},
}

@article {Chatzidakis-generic,
    AUTHOR = {Chatzidakis, Zo\'e},
     TITLE = {Generic automorphisms of separably closed fields},
   JOURNAL = {Illinois J. Math.},
  FJOURNAL = {Illinois Journal of Mathematics},
    VOLUME = {45},
      YEAR = {2001},
    NUMBER = {3},
     PAGES = {693--733},
      ISSN = {0019-2082,1945-6581},
   MRCLASS = {03C60 (03C45 12H10 12L12)},
  MRNUMBER = {1879231},
MRREVIEWER = {G.\ Cherlin},
       URL = {http://projecteuclid.org/euclid.ijm/1258138147},
}

@article {Robinson-DCF,
    AUTHOR = {Robinson, Abraham},
     TITLE = {On the concept of a differentially closed field},
   JOURNAL = {Bull. Res. Council Israel Sect. F},
  FJOURNAL = {Bulletin of the Research Council of Israel. Section F},
    VOLUME = {8F},
      YEAR = {1959},
     PAGES = {113--128},
   MRCLASS = {02.50 (12.80)},
  MRNUMBER = {125016},
MRREVIEWER = {R.\ C.\ Lyndon},
}

@article {McGrail-JSL,
    AUTHOR = {McGrail, Tracey},
     TITLE = {The model theory of differential fields with finitely many
              commuting derivations},
   JOURNAL = {J. Symbolic Logic},
  FJOURNAL = {The Journal of Symbolic Logic},
    VOLUME = {65},
      YEAR = {2000},
    NUMBER = {2},
     PAGES = {885--913},
      ISSN = {0022-4812,1943-5886},
   MRCLASS = {03C60 (03C10 03C45 12H05)},
  MRNUMBER = {1771092},
MRREVIEWER = {G.\ Cherlin},
       DOI = {10.2307/2586576},
       URL = {https://doi.org/10.2307/2586576},
}

@phdthesis{Blum-thesis,
    AUTHOR = {Blum, Lenore Carol},
     TITLE = {Generalized algebraic theories: A model theoretic
              approach},
      school = {Massachusetts Institute of Technology},
 PUBLISHER = {ProQuest LLC, Ann Arbor, MI},
      YEAR = {1968},
   MRCLASS = {99-05},
  MRNUMBER = {2939925},
       URL =
              {\url{http://gateway.proquest.com/openurl?url_ver=Z39.88-2004&rft_val_fmt=info:ofi/fmt:kev:mtx:dissertation&res_dat=xri:pqdiss&rft_dat=xri:pqdiss:0221003}},
}

@article {Yaffe,
    AUTHOR = {Yaffe, Yoav},
     TITLE = {Model completion of {L}ie differential fields},
   JOURNAL = {Ann. Pure Appl. Logic},
  FJOURNAL = {Annals of Pure and Applied Logic},
    VOLUME = {107},
      YEAR = {2001},
    NUMBER = {1-3},
     PAGES = {49--86},
      ISSN = {0168-0072,1873-2461},
   MRCLASS = {03C60 (03C10 12H05)},
  MRNUMBER = {1807840},
       DOI = {10.1016/S0168-0072(00)00025-7},
       URL = {https://doi.org/10.1016/S0168-0072(00)00025-7},
}

@book {Böger-diplom,
    AUTHOR = {B\"{o}ger, Jens},
     TITLE = {Modelltheorie von {K}\"{o}rpern mit paarweise kommutierenden {D}erivationen},
      NOTE = {Diplomarbeit},
    PUBLISHER = {Universit\"{a}t Freiburg},
      YEAR = {1996},
     PAGES = {58 pages},
       URL =  {\url{https://logik.mathematik.uni-freiburg.de/forschung/preprints/Diplom/boeger-dipl.ps}},
}

@book {Bustamante-thesis,
AUTHOR = {Bustamante Medina, Ronald},
TITLE = {Th\'{e}orie des mod\`{e}les des corps diff\'{e}rentiellement clos avec un automorphisme g\'{e}n\'{e}rique},
YEAR = {2005},
NOTE = {Th\`{e}se de doctorat},
PUBLISHER = {Universit\'{e} Paris 7},
PAGES  = {88 pages},
}

@article{Baldwin-Harrington,
title = {Trivial pursuit: Remarks on the main gap},
journal = {Annals of Pure and Applied Logic},
volume = {34},
number = {3},
pages = {209--230},
year = {1987},
issn = {0168-0072},
doi = {10.1016/0168-0072(87)90001-7},
url = {https://www.sciencedirect.com/science/article/pii/0168007287900017},
author = {John T. Baldwin and Leo Harrington}
}

@unpublished{Hrushovski-Sokolović,
    author = {Hrushovski, Ehud and Sokolovi{\'{c}}, {\v{Z}}eljko},
    title = { Minimal subsets of differentially closed fields}, 
    YEAR = {1994} ,
    note = {unpublished manuscript}
}

@article {Pillay-Ziegler-arcs,
    AUTHOR = {Pillay, Anand and Ziegler, Martin},
     TITLE = {Jet spaces of varieties over differential and difference
              fields},
   JOURNAL = {Selecta Math. (N.S.)},
  FJOURNAL = {Selecta Mathematica. New Series},
    VOLUME = {9},
      YEAR = {2003},
    NUMBER = {4},
     PAGES = {579--599},
      ISSN = {1022-1824,1420-9020},
   MRCLASS = {12H05},
  MRNUMBER = {2031753},
MRREVIEWER = {Mykola\ Ya.\ Komarnitski\u i},
       DOI = {10.1007/s00029-003-0339-1},
       URL = {https://doi.org/10.1007/s00029-003-0339-1},
}

@article {HrIt,
    AUTHOR = {Hrushovski, E. and Itai, M.},
     TITLE = {On model complete differential fields},
   JOURNAL = {Trans. Amer. Math. Soc.},
  FJOURNAL = {Transactions of the American Mathematical Society},
    VOLUME = {355},
      YEAR = {2003},
    NUMBER = {11},
     PAGES = {4267--4296},
      ISSN = {0002-9947,1088-6850},
   MRCLASS = {03C60 (12H05)},
  MRNUMBER = {1990753},
MRREVIEWER = {Wai\ Yan\ Pong},
       DOI = {10.1090/S0002-9947-03-03264-1},
       URL = {https://doi.org/10.1090/S0002-9947-03-03264-1},
}

@unpublished{Marker-sm-slides,
    author = {Marker, David},
    title = {Strong minimal sets in differentially closed fields:
Equations of {P}oizat and {L}ienard type},
    note = {\url{https://homepages.math.uic.edu/~marker/IPM-np.pdf}},
    YEAR = {2022},}

@InCollection{Marker-dcf,
 Author = {Marker, David},
 Title = {Model theory of differential fields},
 BookTitle = {Model theory, algebra, and geometry},
 ISBN = {0-521-78068-3},
 Pages = {53--63},
 Year = {2000},
 Publisher = {Cambridge: Cambridge University Press},
 Language = {English},
 URL = {www.msri.org/communications/books/Book39/},
 zbMATH = {1531889},
 Zbl = {0961.03033}
}

@article {Rosenlicht,
    AUTHOR = {Rosenlicht, Maxwell},
     TITLE = {The nonminimality of the differential closure},
   JOURNAL = {Pacific J. Math.},
  FJOURNAL = {Pacific Journal of Mathematics},
    VOLUME = {52},
      YEAR = {1974},
     PAGES = {529--537},
      ISSN = {0030-8730,1945-5844},
   MRCLASS = {12H05 (02H15)},
  MRNUMBER = {352068},
MRREVIEWER = {G.\ Cherlin},
       URL = {http://projecteuclid.org/euclid.pjm/1102911980},
}

@article {vanderput2001,
    AUTHOR = {van der Put, Marius},
     TITLE = {Grothendieck's conjecture for the {R}isch equation
              {$y'=ay+b$}},
   JOURNAL = {Indag. Math. (N.S.)},
  FJOURNAL = {Koninklijke Nederlandse Akademie van Wetenschappen.
              Indagationes Mathematicae. New Series},
    VOLUME = {12},
      YEAR = {2001},
    NUMBER = {1},
     PAGES = {113--124},
      ISSN = {0019-3577,1872-6100},
   MRCLASS = {12H99 (11R58 14H05)},
  MRNUMBER = {1908143},
MRREVIEWER = {Lucia\ Di Vizio},
       DOI = {10.1016/S0019-3577(01)80009-4},
       URL = {https://doi.org/10.1016/S0019-3577(01)80009-4},
}

@article {bostan2024,
    AUTHOR = {Bostan, Alin and Caruso, Xavier and Roques, Julien},
     TITLE = {Algebraic solutions of linear differential equations: an
              arithmetic approach},
   JOURNAL = {Bull. Amer. Math. Soc. (N.S.)},
  FJOURNAL = {American Mathematical Society. Bulletin. New Series},
    VOLUME = {61},
      YEAR = {2024},
    NUMBER = {4},
     PAGES = {609--658},
      ISSN = {0273-0979,1088-9485},
   MRCLASS = {11-02 (05A15 12H25 33C20 34A30 34M15 68W30)},
  MRNUMBER = {4803603},
MRREVIEWER = {Alexander\ B.\ Levin},
       DOI = {10.1090/bull/1835},
       URL = {https://doi.org/10.1090/bull/1835},
}

@article {Cherlin1985,
    AUTHOR = {Cherlin, G. and Harrington, L. and Lachlan, A. H.},
     TITLE = {{$\aleph_0$}-categorical, {$\aleph_0$}-stable structures},
   JOURNAL = {Ann. Pure Appl. Logic},
  FJOURNAL = {Annals of Pure and Applied Logic},
    VOLUME = {28},
      YEAR = {1985},
    NUMBER = {2},
     PAGES = {103--135},
      ISSN = {0168-0072,1873-2461},
   MRCLASS = {03C35 (03C45)},
  MRNUMBER = {779159},
MRREVIEWER = {B.\ I.\ Zil\cprime ber},
       DOI = {10.1016/0168-0072(85)90023-5},
       URL = {https://doi.org/10.1016/0168-0072(85)90023-5},
}

@book {Suer2007,
    AUTHOR = {Suer, Sonat},
     TITLE = {Model theory of differentially closed fields with several
              commuting derivations},
      NOTE = {Thesis (Ph.D.)--University of Illinois at Urbana-Champaign},
 PUBLISHER = {ProQuest LLC, Ann Arbor, MI},
      YEAR = {2007},
     PAGES = {55},
      ISBN = {978-0549-34391-2},
   MRCLASS = {99-05},
  MRNUMBER = {2711111},
       URL =
              {http://gateway.proquest.com/openurl?url_ver=Z39.88-2004&rft_val_fmt=info:ofi/fmt:kev:mtx:dissertation&res_dat=xri:pqdiss&rft_dat=xri:pqdiss:3290394},
}

@article {FS2018,
    AUTHOR = {Freitag, James and Scanlon, Thomas},
     TITLE = {Strong minimality and the {$j$}-function},
   JOURNAL = {J. Eur. Math. Soc. (JEMS)},
  FJOURNAL = {Journal of the European Mathematical Society (JEMS)},
    VOLUME = {20},
      YEAR = {2018},
    NUMBER = {1},
     PAGES = {119--136},
      ISSN = {1435-9855,1435-9863},
   MRCLASS = {03C60 (03C45 12H05 33E30)},
  MRNUMBER = {3743238},
MRREVIEWER = {G.\ Cherlin},
       DOI = {10.4171/JEMS/761},
       URL = {https://doi.org/10.4171/JEMS/761},
}

@unpublished{GR-dcfam,
    author = {Guzy, Nicolas and Rivi\`{e}re, C\'{e}dric},
    title = {On existentially closed partial differential fields with an automorphism},
    YEAR = {2007},
    note = {unpublished manuscript}
}

@article {LS-DCFAm,
    AUTHOR = {Le\'on S\'anchez, Omar},
     TITLE = {On the model companion of partial differential fields with an
              automorphism},
   JOURNAL = {Israel J. Math.},
  FJOURNAL = {Israel Journal of Mathematics},
    VOLUME = {212},
      YEAR = {2016},
    NUMBER = {1},
     PAGES = {419--442},
      ISSN = {0021-2172,1565-8511},
   MRCLASS = {03C60 (12H05)},
  MRNUMBER = {3504332},
MRREVIEWER = {David\ A.\ Pierce},
       DOI = {10.1007/s11856-016-1292-y},
       URL = {https://doi.org/10.1007/s11856-016-1292-y},
}

@article {Hr-MM,
    AUTHOR = {Hrushovski, Ehud},
     TITLE = {The {M}anin-{M}umford conjecture and the model theory of
              difference fields},
   JOURNAL = {Ann. Pure Appl. Logic},
  FJOURNAL = {Annals of Pure and Applied Logic},
    VOLUME = {112},
      YEAR = {2001},
    NUMBER = {1},
     PAGES = {43--115},
      ISSN = {0168-0072,1873-2461},
   MRCLASS = {03C60 (11G10 11U09 12L12)},
  MRNUMBER = {1854232},
MRREVIEWER = {C.\ Toffalori},
       DOI = {10.1016/S0168-0072(01)00096-3},
       URL = {https://doi.org/10.1016/S0168-0072(01)00096-3},
}

@book{Wagner-simple,
 author = {Wagner, Frank O.},
 title = {Simple theories},
 fseries = {Mathematics and its Applications (Dordrecht)},
 series = {Math. Appl., Dordr.},
 issn = {0921-3791},
 volume = {503},
 isbn = {0-7923-6221-7},
 year = {2000},
 publisher = {Dordrecht: Kluwer Academic Publishers},
 language = {English},
 zbMATH = {1425669},
 Zbl = {0948.03032}
}

@article{Casale-Freitag-Nagloo-I,
 author = {Casale, Guy and Freitag, James and Nagloo, Joel},
 title = {Ax-{Lindemann}-{Weierstrass} with derivatives and the genus 0 {Fuchsian} groups},
 fjournal = {Annals of Mathematics. Second Series},
 journal = {Ann. Math. (2)},
 issn = {0003-486X},
 volume = {192},
 number = {3},
 pages = {721--765},
 year = {2020},
 language = {English},
 doi = {10.4007/annals.2020.192.3.2},
 zbMATH = {7285353},
 Zbl = {1465.11094}
}

@incollection{Marker-manin-kernels,
 author = {Marker, David},
 title = {Manin kernels},
 booktitle = {Connections between model theory and algebraic and analytic geometry},
 isbn = {88-7999-312-7},
 pages = {1--21},
 year = {2000},
 publisher = {Rome: Aracne},
 language = {English},
 zbMATH = {1679769},
 Zbl = {1100.14522}
}

@article{KimPillay1997,
  title={Simple theories},
  author={Kim, Byunghan and Pillay, Anand},
  journal={Annals of Pure and Applied Logic},
  volume={88},
  number={2-3},
  pages={149--164},
  year={1997},
  publisher={Elsevier}
}

@article{LS-Zilber-dichotomy,
 author = {Le{\'o}n S{\'a}nchez, Omar},
 title = {Zilber dichotomy for {{\(\mathrm{DCF}_{0, m}\)}}},
 fjournal = {Fundamenta Mathematicae},
 journal = {Fundam. Math.},
 issn = {0016-2736},
 volume = {268},
 number = {3},
 pages = {285--293},
 year = {2025},
 language = {English},
 doi = {10.4064/fm241107-14-1},
 zbMATH = {8028702}
}

@article{BBP-semiabelian,
  title={Semiabelian varieties over separably closed fields, maximal divisible subgroups, and exact sequences},
  author={Benoist, Franck and Bouscaren, Elisabeth and Pillay, Anand},
  journal={Journal of the Institute of Mathematics of Jussieu},
  volume={15},
  number={1},
  pages={29--69},
  year={2016},
  publisher={Cambridge University Press}
}
\end{document}